\documentclass{article}
\usepackage{silence}
\usepackage{amsmath, amsfonts, amsthm, amssymb, amsrefs, amstext}
\usepackage{comment}
\usepackage{xcolor}
\usepackage{faktor}
\usepackage{tabularx}
\usepackage{makecell}
    \setcellgapes{0.25cm}

\theoremstyle{definition}
\newtheorem{definition}{Definition}[subsection]
    
    \newtheorem{example}[definition]{Example}

    \newtheorem*{recall*}{Recall}
\theoremstyle{plain}
    \newtheorem{theorem}[definition]{Theorem}
    \newtheorem{corollary}[definition]{Corollary}
    \newtheorem{proposition}[definition]{Proposition}
    \newtheorem{lemma}[definition]{Lemma}
    
    \newtheorem{conjecture}[definition]{Conjecture}
    \newtheorem*{lemma*}{Lemma}
\theoremstyle{remark}
    \newtheorem{remark}[definition]{Remark}
    \newtheorem*{claim}{Claim}
    
    \newtheorem*{remark*}{Remark}

\newtheoremstyle{appendix}{}{}{\itshape}{}{\bfseries}{.}{.5em}{#1\thmnote{ #3}}
\theoremstyle{appendix}
\newtheorem*{aplemma}{Lemma}

\DeclareMathOperator{\Z}{\mathbb{Z}}
\DeclareMathOperator{\F}{\mathbb{F}}

\newcommand{\legendre}[2]{\genfrac{(}{)}{}{}{#1}{#2}}

\title{Exploration on Incidence Geometry and Sum-Product Phenomena}
\author{Sung-Yi Liao$^1$}
\date{$^1$National Center for Theoretical Sciences, Taiwan}

\begin{document}

\maketitle
\begin{abstract}
In additive combinatorics, Erd\"{o}s-Szemer\'{e}di Conjecture is an important conjecture. It can be applied to many fields, such as number theory, harmonic analysis, incidence geometry, and so on. Additionally, its statement is quite easy to understand, while it is still an open problem. In this dissertation, we investigate the Erd\"{o}s-Szemer\'{e}di Conjecture and its relationship with several well-known results in incidence geometry, such as the Szemer\'{e}di-Trotter Incidence Theorem. We first study these problems in the setting of real numbers and focus on the proofs by Elekes and Solymosi on sum-product estimates. After introducing these theorems, our main focus is the Erd\"{o}s-Szemer\'{e}di Conjecture in the setting of $\mathbb{F}_p$. We aim to adapt several ingenious techniques developed for real numbers to the case of finite fields. Finally, we obtain a result in estimating the number of bisectors over the ring $\mathbb{Z}/p^3\mathbb{Z}$ with $p$ a $4n+3$ prime.
\end{abstract}
\section{Introduction}
\subsection{General settings}
We first introduce the sum set and product set as follows:
\begin{definition}[sum set, product set]
Given a ring $(R,+,\cdot)$ and two non-empty finite subsets $A,B\subseteq R$, define their \textit{sum set}
\[A+B:=\{a+b\mid a\in A,b\in B\}\]
and their \textit{product set}
\[A\cdot B:=\{a\cdot b\mid a\in A,b\in B\}.\]
\end{definition}
It is obvious that the sizes of the sum set $A+B$ and the product set $A\cdot B$ share the same trivial upper bound $|A||B|$ and trivial lower bound $\max\{|A|,|B|\}$ if both $A$ and $B$ contain at least one non-zero-divisor where we use $|A|$ to denote the size of the set $A$. However, if there exists no finite subring, one expects that the sizes of the sum set or product set are much larger than the trivial lower bound. This motivated Erd\"{o}s and Szemer\'{e}di to make their well-known conjecture in the 1980s. Before we go into more details, we first recall the useful Landau notation and Vinogradov notation defined as follows:
\begin{definition}[asymptotic notations]
Throughout this article, all the quantities are nonnegative. The Landau notation and Vinogradov notation are defined as follows: 
\begin{enumerate}
\item We say $f=O(g)$, or equivalently $f\ll g$, if $\lim\limits_{x\to\infty}\frac{f(x)}{g(x)}<\infty$.
\item We say $f=o(g)$ if $\lim\limits_{x\to\infty}\frac{f(x)}{g(x)}=0$.
\item We say $f=\Omega(g)$, or equivalently $f\gg g$, if $\lim\limits_{x\to\infty}\frac{g(x)}{f(x)}<\infty$.
\item We say $f=\omega(g)$ if $\lim\limits_{x\to\infty}\frac{g(x)}{f(x)}=0$.
\item We say $f=\Theta(g)$ if $\lim\limits_{x\to\infty}\frac{f(x)}{g(x)}\in(0,\infty)$.
\item We say $f\sim g$ if $\lim\limits_{x\to\infty}\frac{f(x)}{g(x)}=1$.
\item We say $f\lesssim g$ if there is an absolute $c\in\mathbb{R}$ such that $f(x)\ll g(x)\log^c(x)$.
\item We say $f\approx g$ if $f\lesssim g$ and $g\lesssim f$.
\end{enumerate}
\end{definition}
\subsection{Erd\"{o}s-Szemer\'{e}di Conjecture}
Before we state the Erd\"{o}s-Szemer\'{e}di Conjecture, let us first take a look at the following examples:
\begin{example}
Let $R=\mathbb{R}$.
\begin{enumerate}
\item Consider $A=B=\{1,2,\dotsc,n\}$, an arithmetic progression of length $n$. Then, one can direct compute to see 
\[|A+B|=\left|\{2,3,\dotsc,2n\}\right|=\Theta(|A|),\]
which achieves the lower bound. However, the size of its product set
\[|A\cdot B| \approx |A|^2,\]
which achieves the upper bound. (This result is known as Erd\"{o}s Multiplicative Table Problem and proved by Erd\"{o}s.)
\item On the other hand, consider $A=B=\{2^1,\dotsc,2^n\}$, a geometric progression. We have 
\[|A+B| \approx |A|^2,\]
which achieves the upper bound and its product set
\[|A\cdot B|=\Theta (|A|),\]
which achieves the lower bound.
\end{enumerate}
\end{example}

Based on the above examples, we can see that there are sets whose either sum-set or product-set almost achieves the upper bound. The well-known conjecture of Erd\"{o}s and Szemer\'{e}di states that not only arithmetic progressions or geometric progressions can have either sum-set or product-set almost attain the upper bound but for any finite sets. More precisely, they made the conjecture below.

\begin{conjecture}[Erd\"{o}s-Szemer\'{e}di Conjecture, \cite{Erdos1983sums}]
For any finite subset $A\subseteq\mathbb{Z}$, one has
\[\max\{|A+A|,|A\cdot A|\}\gg_{\delta} {|A|}^{2-\delta}\]
for any constant $\delta>0$.
\end{conjecture}

In order to support their conjecture, they proved the following theorem.

\begin{theorem}[Theorem 1, \cite{Erdos1983sums}]
Given any finite subset $A\subseteq\mathbb{Z}$, we have
\[{|A|}^{1+\epsilon_1}\ll\max\{|A+A|,|A\cdot A|\}\ll {|A|}^2\exp\left(-\epsilon_2\log|A|\log\log|A|\right)\]
for some absolute constants $\epsilon_1,\epsilon_2>0$.
\end{theorem}

Their theorem stimulated many subsequent works that either improved the exponents or generalized it to other settings such as $\mathbb{F}_p$. The conjecture is also widely believed to be true in integers, real numbers, complex numbers, fields, and even arbitrary rings \cite{Tao2009sum}. This conjecture motivates numerous priceless research in various areas and has been extremely important in various research areas. After decades, the conjecture has been connected to other areas, such as Incidence Geometry \cite{Chang2007sum}, Harmonic Analysis \cite{Shkredov2023some}, Number Theory, etc. 
\par

\section{Milestones in Real Number}
In this section, we discuss the Sum-Product Phenomenon in real numbers. Recall that the Erd\"{o}s-Szemer\'{e}di Conjecture in real number is: 
\begin{quote}
For any finite subset $A\subseteq\mathbb{R}$, one has $\max\{|A+A|,|A\cdot A|\}\gg_{\delta} {|A|}^{2-\delta}$ for any $\delta>0$.
\end{quote}
\par
The conjecture is still open, and there have been several important results that are shown in the following table on the next page: 
\begin{table}
\centering
\makegapedcells
\begin{tabularx}{0.9\linewidth}{>{\centering\arraybackslash}X|c}
Authors&Exponential Part\\
\hline
Erd\"{o}s and Szemer\'{e}di (1983, \cite{Erdos1983sums})&$1+\epsilon$\text{ for some }$\epsilon>0$\\
\hline
Elekes (1997, \cite{Elekes1997number})&$1+\frac{1}{4}$\\[3mm]
\hline
Solymosi (2009, \cite{Solymosi2009bounding})&$1+\frac{1}{3}-\delta$ for any $\delta>0$\\[3mm]
\hline
Rudnev and Stevens (2020, \cite{Rudnev2022update})&$1+\frac{1}{3}+\frac{2}{1167}-\delta$ for any $\delta>0$
\end{tabularx}
\end{table}
\newpage
\par
Let us first introduce the beautiful results of Elekes and Solymosi, which are the two most crucial results overall. To start with, we introduce an important notation in Incidence Geometry: Given a set of points $\mathcal{P}\in\mathbb{R}^2$ and a collection $\mathcal{L}$ of lines in $\mathbb{R}^2$, define their incidence number
\[\mathcal{I}(\mathcal{P},\mathcal{L}):=\left|\left\{(p,l)\in\mathcal{P}\times\mathcal{L}\mid p\in l\right\}\right|.\]
\subsection{Elekes' result}\label{2.1}
We first recall the Szemer\'{e}di-Trotter Incidence Theorem, an extremely important theorem in Incidence Geometry (Its proof will be provided in the Appendix A. for the sake of self-containment).
\begin{theorem}[Szemer\'{e}di-Trotter Incidence Theorem, \cite{Szemeredi1983extremal}]\label{Szemeredi-Trotter}
On Euclidean plane $\mathbb{R}^2$, given a point set $\mathcal{P}$ and a line set $\mathcal{L}$, then the number of their incidences
\[\mathcal{I}(\mathcal{P},\mathcal{L})=O\left({|\mathcal{P}|}^{\frac{2}{3}}{|\mathcal{L}|}^{\frac{2}{3}}+|\mathcal{P}|+|\mathcal{L}|\right).\]
\end{theorem}
It is important to remark that the result is sharp,  meaning there exists a configuration of points $P$ and lines $\mathcal{L}$ whose incidence is $\sim |P|^{2/3}|\mathcal{L}|^{2/3}$. In 1997, Elekes brilliantly applied the result of Szemer\'{e}di-Trotter Theorem to the Sum-Product problem and obtained the following theorem:
\begin{theorem}[Elekes, \cite{Elekes1997number}]\label{Elekes}
Given any finite subset $A\subseteq\mathbb{R}$, 
\[\max\{|A+A|,|A\cdot A|\}\gg {|A|}^{1+\frac{1}{4}}.\]
\end{theorem}
\begin{proof}
Consider the point set $\mathcal{P}=(A+A)\times(A\cdot A)$ and the line set 
\[\mathcal{L}=\{l\mid l:y=a_1(x-a_2),a_1,a_2\in A\}.\]
For any lines $l:y=a(x-a')\in\mathcal{L}$ and $a''\in A$, $(a'+a'',aa'')\in\mathcal{P}$, so for every line in $\mathcal{L}$, there are at least $|A|$ points in $\mathcal{P}$ lie on it. In other words, by Theorem \ref{Szemeredi-Trotter}, we have
\begin{equation*}
\begin{split}
&{|A|}^3=|A||\mathcal{L}|\leq\mathcal{I}(\mathcal{P},\mathcal{L})\ll {|A+A|}^{\frac{2}{3}}{|A\cdot A|}^{\frac{2}{3}}{|A|}^{\frac{4}{3}}+|A+A||A\cdot A|+{|A|}^2\\
&\Rightarrow {|A|}^3\ll\max\left\{{|A+A|}^{\frac{2}{3}}{|A\cdot A|}^{\frac{2}{3}}{|A|}^{\frac{4}{3}},|A+A||A\cdot A|\right\}\\
&\Rightarrow {|A|}^{\frac{5}{3}}\ll {|A+A|}^{\frac{2}{3}}{|A\cdot A|}^{\frac{2}{3}}\\
&\Rightarrow {|A|}^{\frac{5}{4}}\ll {|A+A|}^{\frac{1}{2}}{|A\cdot A|}^{\frac{1}{2}}\leq\max\{|A+A|,|A\cdot A|\}.\\
\end{split}
\end{equation*}
This completes the proof.
\end{proof}
\subsection{Solymosi's result}
To demonstrate Solymosi's result, we need the following definition:
\begin{definition}[additive energy, multiplicative energy]
Given two sets $A,B$, their \textit{additive energy} is defined as
\begin{equation}\label{eq1}
E^+(A,B):=|\{(a_1,a_2,b_1,b_2)\in A\times A\times B\times B\mid a_1+b_1=a_2+b_2\}|
\end{equation}
and their \textit{multiplicative energy} is defined as 
\[E^\times(A,B):=|\{(a_1,a_2,b_1,b_2)\in A\times A\times B\times B\mid a_1b_1=a_2b_2\}|.\]
Especially, if $A=B$, define $E^+(A):=E^+(A,A)$ and $E^\times(A):=E^\times(A,A)$.
\end{definition}
Note that the additive energy (resp. multiplicative energy) indicates how many redundant additive pairs (resp. multiplicative pairs) the sets have. In other words, a set with large additive energy (resp. multiplicative energy) will have a small sum set (resp. product set). This makes these energy estimates useful for estimating sum and product sets. To be more specific, define
\[r_{A+A}(z):=|\{(a_1,a_2)\in A\times A\mid z=a_1+a_2\}|.\]
Then, by Cauchy-Schwarz inequality, we can see that
\begin{equation}\label{eq2}
{|A|}^4={\left(\sum_{z\in A+A}r_{A+A}(z)\right)}^2\leq|A+A|\sum_{z\in A+A}{r_{A+A}(z)}^2=|A+A|E^+(A).
\end{equation}
Therefore, we can obtain a lower bound of the sum set by an upper bound of additive energy.
\begin{theorem}[Solymosi, \cite{Solymosi2009bounding}]
Given any finite subset $A\subseteq\mathbb{R}$, 
\[\max\{|A+A|,|A\cdot A|\}\gg {|A|}^{1+\frac{1}{3}}.\]
\end{theorem}
\begin{proof}
Without loss of generality, we may assume that every element in $A$ is positive. Let 
\[r_{A/A}(x):=\{(a_1,a_2)\in A\times A\mid x=a_1/a_2\},xA:=\{xa\mid a\in A\}\]
for any element $x$ and 
\[A/A:=\{a/a'\mid a,a'\in A\}.\]
First, we have 
\begin{equation*}
\begin{split}
E^\times(A)&=\sum_{x\in A/A}{r_{A/A}(x)}^2=\sum_{x\in A/A}{|xA\cap A|}^2\\
&=\sum_{i=1}^{\lceil\log_2|A|\rceil}\sum_{\substack{x:2^{i-1}\leq|xA\cap A|<2^i}}{|xA\cap A|}^2.
\end{split}
\end{equation*}
\par 
Then, by the pigeonhole principle, there is an index $1\leq i_0\leq\lceil\log_2|A|\rceil$ such that
\begin{equation}\label{eq3}
\frac{E^\times(A)}{\lceil\log_2|A|\rceil}\leq\sum_{x:2^{i_0-1}\leq|xA\cap A|<2^{i_0}}{|xA\cap A|}^2.
\end{equation}
Let $D:=\{x\mid 2^{i_0-1}\leq|xA\cap A|<2^{i_0}\}=\{s_1,\dotsc,s_n\}$ with $|D|=n$ and $s_1<\dotsb<s_n$. Set the line $l_i:y=s_ix$ for $i=1,\dotsc,n$ and $l_{n+1}:x=a_1$ where $a_1$ is the smallest element in $A$. For $i=1,\dotsc,n$, consider the point sets
\[P_i:=\big(l_i\cap(A\times A)\big)+\big(l_{i+1}\cap(A\times A)\big)\text{ and }P:=\bigsqcup_{i=1}^nP_i\subseteq(A\times A)+(A\times A)\]
where the addition symbol means vector addition and $A\times A$ means the Cartesian product.
\par
Next, since $|l_i\cap(A\times A)|\geq2^{i_0-1}$,
\[|P|=\sum_{i=1}^n|P_i|\geq n2^{i_0}.\]
On the other hand, 
\[|P|\leq|(A+A)\times(A+A)|={|A+A|}^2\Rightarrow n2^{i_0}\leq {|A+A|}^2.\]
Together with Equation \eqref{eq3}, we get
\[\frac{E^\times(A)}{\lceil\log_2|A|\rceil}\leq n2^{i_0}\leq {|A+A|}^2.\]
Combining it with the multiplicative version of Equation \eqref{eq2}, we get
\[\frac{{|A|}^4}{|A\cdot A|\lceil\log_2|A|\rceil}\leq\frac{E^\times(A)}{\lceil\log_2|A|\rceil}\leq {|A+A|}^2.\]
Hence, 
\[\frac{{|A|}^{\frac{4}{3}}}{{\lceil\log_2|A|\rceil}^{\frac{1}{3}}}\leq {|A+A|}^{\frac{2}{3}}{|A\cdot A|}^{\frac{1}{3}}\leq\max\{|A+A|,|A\cdot A|\},\]
which completes this proof.
\end{proof}
Note that in the above proof, we use the pigeonhole principle to choose a suitable index $i_0$ that satisfies Equation \eqref{eq3}. This skill, called the ``dyadic argument,'' is widely used in the proofs of many results of the Sum-Product Phenomena \cite{Tao2006additive}.
\section{Connection between Incidence Geometry and Sum-Product Phenomena in \texorpdfstring{$\mathbb{F}_p$}{}}
In this section, we will display several useful Incidence Geometry results in $\mathbb{F}_p$ and especially focus on those that have strong connections with Sum-Product estimates. In fact, the known result by Bourgain, Katz, and Tao has established the first non-trivial bound for the point-line incidences in the setting of $\mathbb F_p$, which was an application of their sum-product estimate. However, there have been many improvements since then, which we state below.
\subsection{Szemer\'{e}di-Trotter type theorem}
We start with a familiar example that we mentioned in Section \ref{2.1}. In the proof of Theorem \ref{Elekes}, the most important lemma is Theorem \ref{Szemeredi-Trotter} (Szemer\'{e}di-Trotter Theorem). Therefore, obtaining Szemer\'{e}di-Trotter type theorem in $\mathbb{F}_p$, where $p$ is a prime, can also provide a non-trivial bound for Erd\"{o}s-Szemer\'{e}di Conjecture. Sophie Stevens and Frank De Zeeuw obtain an explicit result as follows \cite{Stevens2017improved}. (Proof of the following theorem will be provided in Appendix for the sake of self-containment.)
\begin{theorem}[Szemer\'{e}di-Trotter type Theorem in $\mathbb{F}_p$, \cite{Stevens2017improved}]\label{Szemeredi-Trotter_in_Fp}
On $\mathbb{F}_p^2$ where $p$ is a prime, given a point set $\mathcal{P}=A\times B$ and a line set $\mathcal{L}$ with 
\[|A|\leq|B|,|A|{|B|}^2\leq {|\mathcal{L}|}^3\text{, and }|A||\mathcal{L}|\ll p^2\]
then the number of their incidences
\[\mathcal{I}(A\times B,\mathcal{L})=O\left({|A|}^{\frac{3}{4}}{|B|}^{\frac{1}{2}}{|\mathcal{L}|}^{\frac{3}{4}}+|\mathcal{L}|\right).\]
\end{theorem}
With this theorem, we can deduce the following theorem via an argument similar to the proof of Theorem \ref{Elekes}.
\begin{corollary}[Corollary 9, \cite{Stevens2017improved}]\label{Stevens}
Given any subset $A\subseteq\mathbb{F}_p$ with $|A|\ll p^{\frac{2}{3}}$, 
\[\max\{|A+A|,|A\cdot A|\}\gg {|A|}^{1+\frac{1}{5}}.\]
\end{corollary}
\begin{proof}
For a subset $A$, divide it into two cases:
\begin{enumerate}
\item If $|A+A|\leq|A\cdot A|$, then set 
\[\mathcal{L}:=\{l:y=a(x-a')\mid a,a'\in A\},A'=A+A\text{, and }B'=A\cdot A.\]
We will find that $A', B'$, and $\mathcal{L}$ satisfy the assumption of Theorem \ref{Szemeredi-Trotter_in_Fp}. Additionally, for any line $l:y=a(x-a')$ in $\mathcal{L}$ and $a''\in A$, 
\[(a'+a'',aa'')\in(A'\times B')\cap l.\]
In other words, by Theorem \ref{Szemeredi-Trotter_in_Fp},
\begin{equation*}
\begin{split}
&{|A|}^3=|A||\mathcal{L}|\leq I(A'\times B',\mathcal{L})\ll {|A+A|}^{\frac{3}{4}}{|A\cdot A|}^{\frac{1}{2}}{|A|}^{\frac{3}{2}}+{|A|}^2\\
&\Rightarrow {|A|}^6\ll {|A+A|}^3{|A\cdot A|}^2\\
&\Rightarrow {|A|}^{1+\frac{1}{5}}\ll\max\{|A+A|,|A\cdot A|\}.
\end{split}
\end{equation*}
\item On the other hand, if $|A+A|>|A\cdot A|$, then set 
\[\mathcal{L}:=\{l:a(y-a')=x\mid a,a'\in A\},A'=A\cdot A\text{, and }B'=A+A.\]
We can see that $\mathcal{L}, A', B'$ satisfy the assumptions of Theorem \ref{Szemeredi-Trotter_in_Fp}. Moreover, for any line $l:a(y-a')=x$ in $\mathcal{L}$ and $a''\in A$, 
\[(aa'',a'+a'')\in(A'\times B')\cap l.\]
That is, 
\begin{equation*}
\begin{split}
&{|A|}^3=|A||\mathcal{L}|\leq\mathcal{I}(A'\times B',\mathcal{L})\ll {|A\cdot A|}^{\frac{3}{4}}{|A+A|}^{\frac{1}{2}}{|A|}^{\frac{3}{2}}+{|A|}^2\\
&\Rightarrow {|A|}^6\ll {|A\cdot A|}^3{|A+A|}^2\\
&\Rightarrow {|A|}^{1+\frac{1}{5}}\ll\max\{|A+A|,|A\cdot A|\}.
\end{split}
\end{equation*}
\end{enumerate}
To sum up, in both cases, our statement holds.
\end{proof}
\begin{corollary}\label{Stevens2}
In the above proof, we can see that 
\[{|A|}^6\ll\min\{{|A+A|}^3{|A\cdot A|}^2, {|A\cdot A|}^3{|A+A|}^2\}.\]
Thus, in the extremal cases, we have the following bounds:
\begin{enumerate}
\item If $|A+A|=\Theta(|A|)$, then $|A\cdot A|\gg {|A\cdot A|}^{1+\frac{1}{2}}$.
\item If $|A\cdot A|=\Theta(|A|)$, then $|A+A|\gg {|A+A|}^{1+\frac{1}{2}}$.
\end{enumerate}
\end{corollary}
\subsection{Estimation of a special kind of energy}
In this part, we will focus on the estimation of a special kind of energy, which is first considered in \cite{Murphy2015variations}. After the estimation, we will see that this energy will provide another brilliant bound for Erd\"{o}s-Szemer\'{e}di Conjecture.
\begin{definition}[variant-slope set, $r_Q$, restricted variant-slope energy]
Let $Z,A_1,A_2$ be subsets of $\mathbb{F}_p$. Define \textit{variant-slope set} 
\[Q:=(A_1+A_2)/(A_1+A_2)=\left\{\frac{a_1+a_2}{a_1'+a_2'}\,\middle|\,a_1,a_1'\in A_1, a_2,a_2'\in A_2\right\},\]
a function
\[r_Q(z):=\left\{(a_1,a_1',a_2,a_2')\in A_1\times A_1\times A_2\times A_2\,\middle|\,z=\frac{a_1+a_2}{a_1'+a_2'}\right\},\]
and the \textit{restricted variant-slope energy}
\begin{equation*}
\begin{split}
R(Z,A_1,A_2):=&\sum_{z\in Z}{r_Q(z)}^2\\
=&\text{ the number of solutions to }``z=\frac{a_{1,1}+a_{2,1}}{a_{1,2}+a_{2,2}}=\frac{a_{1,3}+a_{2,3}}{a_{1,4}+a_{2,4}}"\\
&\text{ where }z\in Z,a_{i,j}\in A_i\text{ for all }i=1,2,j=1,2,3,4.
\end{split}
\end{equation*}
\end{definition}
Note that subtractive energy, defined as 
\[E^-(A):=\{a_1,a_2,a_3,a_4\in A\times A\times A\times A\mid a_1-a_2=a_3-a_4\},\]
is as same as additive energy, so a subtraction acts similar way as an addition from the perspective of energy. Also, notice that the collection of slopes of lines passing both $A_1\times A_1$ and $A_2\times A_2$ is 
\[\frac{A_1-A_2}{A_1-A_2}=\left\{\frac{a_1-a_2}{a_1'-a_2'}\,\middle|\,a_1,a_1'\in A_1,a_2,a_2'\in A_2\right\},\]
which is a dual version of $Q=(A_1+A_2)/(A_1+A_2)$. In other words, $Q$ is the collection of variant slopes from the energy's point of view. Furthermore, compared to Equation \eqref{eq1}, when we take $Z=(A_1-A_2)/(A_1-A_2)$ or $Z=\mathbb{F}_p$, the definition of $R(Z,A_1,A_2)$ is in energy form. For general $Z$, since the value of variant-slope is restricted in $Z$, we call $R(Z, A_1, A_2)$ restricted variant-slope energy, and it gives us a non-trivial bound for Erd\"{o}s-Szemer\'{e}di Conjecture in $\mathbb{F}_p$.
\par
First, without loss of generality, assume $|A_1|\leq|A_2|$. Then, there is a trivial upper bound
\[R(Z,A_1,A_2)\leq R(\mathbb{F}_p,A_1,A_2)\leq {|A_1|}^4{|A_2|}^3.\]
\par 
Next, we want to improve the upper bound of $R(Z, A_1, A_2)$. Boqing Xue (2021, \cite{Xue2021asymmetric}) solved a similar case in the real number for this question. Following the pace of Xue, Dung (2022, \cite{Dung2022reu}) proved the following two upper bounds in $\mathbb{F}_p$:
\begin{lemma}[\cite{Dung2022reu}]\label{Dung}
Let $Z,A_1,A_2\subseteq\mathbb{F}_p$ with $|A_1|\leq|A_2|,|Z|\leq {|A_1|}^2$, and ${|A_1|}^3\ll p^2$. Then, 
\begin{equation}\label{eq4}
\sum_{z\in Z}r_Q(z)\lesssim {|Z|}^{\frac{1}{2}}{|A_1|}^{\frac{33}{16}}{|A_2|}^{\frac{5}{4}}.
\end{equation}
Additionally, we have
\begin{equation}\label{eq5}
R(Z,A_1,A_2)\lesssim {|A_1|}^{\frac{33}{8}}{|A_2|}^{\frac{5}{2}}.
\end{equation}
\end{lemma}
\begin{proof}
To prove this claim, note that
\begin{equation}\label{eq6}
\begin{split}
&\sum_{z\in Z}r_Q(z)\\
&=\sum_{z\in Z}\sum_{y\in\mathbb{F}_p}\left|\left\{(a_1,a_1',a_2,a_2')\,\middle|\,a_1-a_1'z=a_2'z-a_2=y,a_i,a_i'\in A_i\right\}\right|\\
&=\sum_{(z,y)\in\mathcal{P}}r_1(z,y)r_2(z,y)
\end{split}
\end{equation}
where 
\[\mathcal{P}:=\{(z,y)\mid z\in Z,y\in(A_1-zA_1)\cap(A_2-zA_2)\},\]
\[r_1(z,y):=\left|\left\{(a,a')\in A_1\times A_1\mid a-a'z=y\right\}\right|,\]
and
\[r_2(z,y):=\left|\left\{(a,a')\in A_2\times A_2\mid a'z-a=y\right\}\right|.\]
\par 
Next, we split the summation in Equation \eqref{eq6} into four parts $S_{1,1},S_{2,1},S_{2,2}$, and $S_{1,2}$ with 
\[S_{j_1,j_2}:=\sum_{\substack{(z,y)\in\mathcal{P}\\r_1(z,y)\in I_{j_1}\\r_2(z,y)\in I_{j_2}}}r_1(z,y)r_2(z,y)\]
where $j_1,j_2\in\{1,2\}$ and $I_1=\{1\},I_2=\{2,3,\dotsc\}$. Now, we will bound $S_{1,1},S_{2,1},S_{2,2}$, and $S_{1,2}$, separately:
\begin{enumerate}
\item For $S_{1,1}$, we know that
\begin{equation}\label{eq7}
S_{1,1}\leq|\mathcal{P}|\leq|Z|{|A_1|}^2.
\end{equation}
\item For $S_{2,1}$, we obtain
\begin{equation}\label{eq8}
\begin{split}
S_{2,1}&\leq\sum_{(z,y)\in\mathcal{P}}r_1(z,y)\leq|\{(z,y,a_1,a_1')\in Z\times\mathbb{F}_p\mid a_1-a_1'z=y\}|\\
&=|Z|{|A_1|}^2.
\end{split}
\end{equation}
\item For $S_{2,2}$, applying H\"{o}lder inequality several times, we see that
\begin{equation}\label{eq9}
\begin{split}
S_{2,2}&=\sum_{\substack{(z,y)\in\mathcal{P}\\r_1(z,y),r_2(z,y)\geq2}}r_1(z,y)r_2(z,y)\\
&\leq {\left(\sum_{\substack{(z,y)\in\mathcal{P}\\r_1(z,y),r_2(z,y)\geq2}}{r_1(z,y)}^{\frac{4}{3}}\right)}^{\frac{3}{4}}{\left(\sum_{\substack{(z,y)\in\mathcal{P}\\r_2(z,y)\geq2}}{r_2(z,y)}^4\right)}^{\frac{1}{4}}\\
&\leq {\left(\sum_{(z,y)\in\mathcal{P}}r_1(z,y)\right)}^{\frac{1}{2}}{\left(\sum_{(z,y)\in\mathcal{P}}{r_1(z,y)}^2\right)}^{\frac{1}{4}}{\left(\sum_{(z,y)\in\mathcal{P}}{r_2(z,y)}^4\right)}^{\frac{1}{4}}
\end{split}
\end{equation}
Next, we separately estimate these three parts:
\begin{enumerate}
\item For the first part, we have
\begin{equation}\label{eq10}
\sum_{(z,y)\in\mathcal{P}}r_1(z,y)\leq |Z|{|A_1|}^2.
\end{equation}
\item Secondly, for $a\in A_1$, we define
\[r_{1,a}(z):=\left|\left\{(a_1,a_1',a_1'')\in A\times A\times A\,\middle|\,a_1''-a_1'=z(a_1-a)\right\}\right|.\]
Then, 
\[\sum_{y\in\mathbb{F}_p}{r_1(z,y)}^2=\sum_{a\in A_1}r_{1,a}(z).\]
According to the relation between $|A_1|$ and $|Z|$, divide it into two cases:
\begin{enumerate}
\item Assume $|A_1|\leq|Z|$:
\par
For a fixed $a\in A_1$, let $\mathcal{L}_a$ be the collection of lines of the form $l:a_1'-x=y(a-a_1)$ where $(a_1,a_1')\in A_1\times A_1$. Set $A:=A_1,B:=Z,\mathcal{L}:=\mathcal{L}_a$ and notice that $A,B,\mathcal{L}_a$ satisfy the assumptions of Theorem \ref{Szemeredi-Trotter_in_Fp}. Thus, 
\begin{equation}\label{eq11}
\begin{split}
\sum_{(z,y)\in\mathcal{P}}{r_1(z,y)}^2&=\sum_{z\in Z}\sum_{a\in A_1}r_{1,a}(z)\leq\sum_{a\in A_1}\mathcal{I}(A_1\times Z,\mathcal{L}_a)\\
&\ll\sum_{a\in A_1}{|A_1|}^{\frac{3}{4}}{|Z|}^{\frac{1}{2}}{|A_1|}^{\frac{3}{2}}+{|A_1|}^2\\
&\ll {|Z|}^{\frac{1}{2}}{|A_1|}^{\frac{13}{4}}.
\end{split}
\end{equation}
\item On the other hand, assume $|A_1|\geq|Z|$:
\par
The proving technique is similar to the above. But this time, define $\mathcal{L}_a$ as the family of lines in the form 
$l:a_1'-y=x(a_1-a)$ where $a_1,a_1'\in A_1$ and set $A:=Z,B:=A_1,\mathcal{L}:=\mathcal{L}_a$. Applying Theorem \ref{Szemeredi-Trotter_in_Fp}, we obtain 
\begin{equation}\label{eq12}
\begin{split}
\sum_{(z,y)\in\mathcal{P}}{r_1(z,y)}^2&\leq\sum_{z\in Z}\sum_{a\in A_1}r_{1,a}(z)\leq\sum_{a\in A_1}\mathcal{I}(Z\times A_1,\mathcal{L}_a)\\
&\ll\sum_{a\in A_1}{|Z|}^{\frac{3}{4}}{|A_1|}^{\frac{1}{2}}{|A_1|}^{\frac{3}{2}}+{|A_1|}^2\\
&\ll {|Z|}^{\frac{3}{4}}{|A_1|}^3\leq {|Z|}^{\frac{1}{2}}{|A_1|}^{\frac{13}{4}}.
\end{split}
\end{equation}
\end{enumerate}
Hence, by Equation \eqref{eq11} and \eqref{eq12}, we have
\begin{equation}\label{eq13}
\sum_{(z,y)\in\mathcal{P}}{r_1(z,y)}^2\ll {|Z|}^{\frac{1}{2}}{|A_1|}^{\frac{13}{4}}.
\end{equation}
\item As for the last part, let $\mathcal{L}=\{l_{z_0,y_0}:y=z_0x+y_1\mid(z_0,y_0)\in\mathcal{P}\}$. For $2\leq k\leq|A_i|$, let $\mathcal{L}_{i,k}$ be the set of lines in $\mathcal{L}$ that contains at least $k$ points in $A_i\times A_i$. Using Theorem \ref{Szemeredi-Trotter_in_Fp}, we obtain
\begin{equation*}
\begin{split}
&k|\mathcal{L}_{i,k}|\leq\mathcal{I}(A_i\times A_i,\mathcal{L})\ll {|A_i|}^{\frac{3}{4}}{|A_i|}^{\frac{1}{2}}{|\mathcal{L}_{i,k}|}^{\frac{3}{4}}+|\mathcal{L}_{i,k}|\\
&\Rightarrow|\mathcal{L}_{i,k}|\ll\frac{{|A_i|}^5}{k^4}.
\end{split}
\end{equation*}
Thus, by the above estimate, we get
\begin{equation}\label{eq14}
\begin{split}
\sum_{\substack{(z,y)\in\mathcal{P}\\r_2(z,y)\geq2}}{r_2(z,y)}^4&=\sum_{2\leq k\leq|A_2|}k^4\left|\left\{(z,y)\in\mathcal{P}\mid r_2(z,y)=k\right\}\right|\\
&\ll\sum_{2\leq k\leq|A_2|}k^3|\mathcal{L}_{2,k}|\ll\sum_{2\leq k\leq|A_2|}k^{-1}{|A_2|}^5\\
&\lesssim {|A_2|}^5.
\end{split}
\end{equation}
\end{enumerate}
Combining Equation \eqref{eq9}, \eqref{eq10},\eqref{eq13}, and \eqref{eq14}, we know that 
\begin{equation}\label{eq15}
S_{2,2}\lesssim {\left(|Z|{|A_1|}^2\right)}^{\frac{1}{2}}{\left({|Z|}^{\frac{1}{2}}{|A_1|}^{\frac{13}{4}}\right)}^{\frac{1}{4}}{\left({|A_2|}^5\right)}^{\frac{1}{4}}={|Z|}^{\frac{5}{8}}{|A_1|}^{\frac{29}{16}}{|A_2|}^{\frac{5}{4}}.
\end{equation}
\item For $S_{1,2}$, applying Cauchy-Schwarz inequality twice and Equation \eqref{eq14}, we get
\begin{equation}\label{eq16}
\begin{split}
S_{1,2}&\leq\sum_{\substack{(z,y)\in\mathcal{P}\\r_2(z,y)\geq2}}r_2(z,y)\leq {|\mathcal{P}|}^{\frac{1}{2}}{\left(\sum_{\substack{(z,y)\in\mathcal{P}\\r_2(z,y)\geq2}}{r_2(z,y)}^2\right)}^{\frac{1}{2}}\\
&\leq {|\mathcal{P}|}^{\frac{3}{4}}{\left(\sum_{\substack{(z,y)\in\mathcal{P}\\r_2(z,y)\geq2}}{r_2(z,y)}^4\right)}^{\frac{1}{4}}\lesssim {|Z|}^{\frac{3}{4}}{|A_1|}^{\frac{3}{2}}{|A_2|}^{\frac{5}{4}}.
\end{split}
\end{equation}
\end{enumerate}
By the estimate of Equation \eqref{eq7}, \eqref{eq8}, \eqref{eq15}, and \eqref{eq16}, we get
\begin{equation}\label{eq17}
\begin{split}
\sum_{z\in Z}r_Q(z)&=S_{1,1}+S_{1,2}+S_{2,1}+S_{2,2}\\
&\lesssim|Z|{|A_1|}^2+|Z|{|A_1|}^2+{|Z|}^{\frac{5}{8}}{|A_1|}^{\frac{29}{16}}{|A_2|}^{\frac{5}{4}}+{|Z|}^{\frac{3}{4}}{|A_1|}^{\frac{3}{2}}{|A_2|}^{\frac{5}{4}}\\
&\lesssim {|Z|}^{\frac{5}{8}}{|A_1|}^{\frac{29}{16}}{|A_2|}^{\frac{5}{4}}\lesssim {|Z|}^{\frac{1}{2}}{|A_1|}^{\frac{33}{16}}{|A_2|}^{\frac{5}{4}}
\end{split}
\end{equation}
This completes the proof of Equation \eqref{eq4}. 
\par
Next, to prove Equation \eqref{eq5}, define 
\[Z_t:=\{z\in Z\mid r_Q(z)\geq t\}\text{ for any }t\geq1.\]
Note that 
\[|Z_t|\leq|Z|\leq {|A_1|}^2.\]
Replacing $Z$ with $Z_t$ and applying \eqref{eq17}, we have
\[t|Z_t|\leq\sum_{z\in Z_t}r_Q(z)\ll {|Z_t|}^{\frac{1}{2}}{|A_1|}^{\frac{33}{16}}{|A_2|}^{\frac{5}{4}}.\]
That is, 
\[|Z_t|\ll\frac{{|A_1|}^{\frac{33}{8}}{|A_2|}^{\frac{5}{2}}}{t^2}.\]
As a result, 
\begin{equation*}
\begin{split}
R(Z,A_1,A_2)&=\sum_{z\in Z}{r_Q(z)}^2=\sum_{t\leq {|A_1|}^2|A_2|}t^2|\{z\mid r_Q(z)=t\}|\\
&\ll\sum_{t\leq {|A_1|}^2|A_2|}t|Z_t|\ll\sum_{t\leq {|A_1|}^2|A_2|}t^{-1}{|A_1|}^{\frac{33}{8}}{|A_2|}^{\frac{5}{4}}\\
&\lesssim {|A_1|}^{\frac{33}{8}}{|A_2|}^{\frac{5}{4}},
\end{split}
\end{equation*}
which completes the proof of Equation \eqref{eq5}.
\end{proof}
With Lemma \ref{Dung}, we can start to estimate the relation between the sum set and the product set. The following results are also motivated by \cite{Murphy2015variations},\cite{Stevens2017improved}, and \cite{Xue2021asymmetric}. To be more specific, Xue studied Erd\"{o}s-Szemer\'{e}di Conjecture in the real number while Stevens and De Zeeuw proved several useful Incidence Geometry results in arbitrary fields. Thus, we combine their works to obtain Proposition \ref{third_additive_energy_estimate} and Theorem \ref{my_first_result}, which are the $\mathbb{F}_p$-version of Proposition 3.2 and Theorem 1.6 in \cite{Xue2021asymmetric}.
\begin{definition}[$n$th-order energy]
Previously, the energy we considered, which is called second-order energy, is equal to 
\begin{equation*}
\begin{split}
E^+(A,B)&:=\left|\left\{(a,a',b,b')\in A\times A\times B\times B\mid a-a-b=a'-b'\right\}\right|\\
&=\sum_{z\in A-B}{r_{A-B}(z)}^2
\end{split}
\end{equation*}
where $r_{A-B}(z):=\left|\{(a,b)\in A\times B\mid z=a-b\}\right|$. Now, we generalize this concept into \textit{$n$th-order energy}, which is defined as follows:
\begin{equation*}
\begin{split}
E_n^+(A,B)&:=\left|\left\{(a_1,\dotsc,a_n,b_1,\dotsc,b_n)\in A^n\times B^n\mid a_1-b_1=\dotsb=a_n-b_n\right\}\right|\\
&=\sum_{z\in A-B}{r_{A-B}(z)}^n.
\end{split}
\end{equation*}
\end{definition}
\begin{proposition}\label{third_additive_energy_estimate}
Let $A,B\subseteq\mathbb{F}_p$. Suppose that $|A|\lesssim|B|\ll p^{\frac{2}{5}}$. Then, 
\[{E_3^+(A,B)}^{\frac{4}{3}}{|B|}^{-4}\lesssim|A\cdot A|.\]
\end{proposition}
To prove this proposition, we will need the following Lemmas:
\begin{lemma}\label{lemma3.2.5}
Let $A,B\subseteq\mathbb{F}_p$ with $|A|\leq|B|\ll p^{\frac{2}{5}}$ and 
\[\mathcal{L}:=\{l_{a,b}:y=ax+b\mid(a,b)\in A\times B\}.\]
Then, for any $\mathcal{P}\subseteq\mathbb{F}_p^2$ with $|A|{|B|}^2\leq {|\mathcal{P}|}^3$ and $|A||\mathcal{P}|\ll p^2$, we have
\[\mathcal{I}(\mathcal{P},\mathcal{L})\ll {|A|}^{\frac{3}{4}}{|B|}^{\frac{1}{2}}{|\mathcal{P}|}^{\frac{3}{4}}+|\mathcal{P}|.\]
\end{lemma}
\begin{proof}
Define $\mathcal{P}':=A\times B$ and $\mathcal{L}':=\{l_{c,d}:-y=cx-d\mid (c,d)\in\mathcal{P}\}$. Then, $I(\mathcal{P},\mathcal{L})=I(\mathcal{P}',\mathcal{L}')$. Thus, by theorem \ref{Szemeredi-Trotter_in_Fp}, we know that 
\[\mathcal{I}(\mathcal{P},\mathcal{L})=\mathcal{I}(\mathcal{P}',\mathcal{L}')\ll {|A|}^{\frac{3}{4}}{|B|}^{\frac{1}{2}}{|\mathcal{P}|}^{\frac{3}{4}}+|\mathcal{P}|.\]
\end{proof}
\begin{lemma}\label{lemma3.2.6}
Let $A,B,X\subseteq\mathbb{F}_p$ such that $|X|\leq|A||B|$ and $|A|\leq|B|\ll p^{\frac{2}{5}}$. Then, 
\[\sum_{x\in X}E^+(A,xB)\ll {|A|}^{\frac{5}{3}}{|B|}^{\frac{4}{3}}{|X|}^{\frac{2}{3}}.\]
\end{lemma}
\begin{proof}
Note that 
\begin{equation}\label{eq18}
\sum_{x\in X}E^+(A,xB)=\sum_{x\in X}\sum_{y\in\mathbb{F}_p}r_{A+xB}^2(y).
\end{equation}
\begin{claim}
Let $R_t:=\{(x,y)\mid r_{A+xB}(y)\geq t\}$. Then, for any integer $2\leq t\leq|A|$, 
\[|R_t|\ll\frac{{|A|}^3{|B|}^2}{t^4}.\]
\end{claim}
Note that for $t>|A|$, $R_t=\emptyset$. Define a collection of lines $\mathcal{L}:=\{l_{a,b}:y=ax+b\mid (a,b)\in A\times B\}$. Since $r_{A+xB}(y)=|\{(a,b)\mid y=ax+b\}|$. Thus, for every pair $(x,y)\in R_t$, $|\{l_{a,b}\mid (x,y)\in l_{a,b}\}|\geq t$. Divide it into two cases:
\begin{enumerate}
\item Assume $|A|{|B|}^2\leq {|R_t|}^3$ so that we can apply lemma \ref{lemma3.2.5}. By lemma \ref{lemma3.2.5}, 
\begin{equation*}
\begin{split}
&t|R_t|\leq\mathcal{I}(R_t,\mathcal{L})\ll {|A|}^{\frac{3}{4}}{|B|}^{\frac{1}{2}}{|R_t|}^{\frac{3}{4}}+|R_t|\\
&\Rightarrow|R_t|\ll\frac{{|A|}^3{|B|}^2}{t^4}+\frac{|R_t|}{t}\ll\frac{{|A|}^3{|B|}^2}{t^4}
\end{split}
\end{equation*}
since $|R_t|\leq {|A|}^2{|B|}^2$. This completes the proof of the claim. 
\item Assume ${|R_t|}^3\leq |A|{|B|}^2$. Then, via this assumption,
\[|R_t|\leq\underbrace{{|A|}^{\frac{1}{3}}{|B|}^{\frac{2}{3}}\times\frac{t^4}{{|A|}^3{|B|}^2}}_{=O(1)} \times\frac{{|A|}^3{|B|}^2}{t^4}\ll\frac{{|A|}^3{|B|}^2}{t^4},\]
which also completes this proof of the claim.
\end{enumerate}
\par 
Let $\Delta$ be a parameter to be determined later. Then, by equation \eqref{eq18},
\begin{equation*}
\begin{split}
\sum_{x\in X}E^+(A,xB)&=\sum_{x\in X}\sum_{y\in\mathbb{F}_p}r_{A+xB}^2(y)\\
&=\sum_{\substack{(x,y)\in X\times\mathbb{F}_p\\r_{A+xB}(y)\leq\Delta}}r_{A+xB}^2(y)+\sum_{\substack{(x,y)\in X\times\mathbb{F}_p\\r_{A+xB}(y)>\Delta}}r_{A+xB}^2(y)
\end{split}
\end{equation*}
\begin{enumerate}
\item For the first term, observe that 
\[\sum_{\substack{(x,y)\in\mathbb{F}_p\\r_{A+xB}(y)\leq\Delta}}r_{A+xB}^2(y)\leq\Delta\sum_{x\in X}\sum_{y\in\mathbb{F}_p}r_{A+xB}(y)=\Delta|A||B||X|.\]
\item For the second term, by the above claim, we have 
\begin{equation*}
\begin{split}
\sum_{\substack{(x,y)\in X\times\mathbb{F}_p\\r_{A+xB}(y)>\Delta}}r_{A+xB}^2(y)&=\sum_{j\geq1}\sum_{\substack{(x,y)\in X\times\mathbb{F}_p\\2^{j-1}\Delta<r_{A+xB}(y)\leq 2^j\Delta}}r_{A+xB}^2(y)\\
&\ll\sum_{j\geq1}\left(\frac{{|A|}^3{|B|}^2}{2^{4j-4}\Delta^4}\right)2^{2j}\Delta^2\ll\frac{{|A|}^3{|B|}^2}{\Delta^2}.
\end{split}
\end{equation*}
\end{enumerate}
\par
Let $\Delta\sim {|A|}^{\frac{2}{3}}{|B|}^{\frac{1}{3}}{|X|}^{-\frac{1}{3}}$. Then, we know that 
\[\sum_{x\in X}E^+(A,xB)\ll {|A|}^{\frac{5}{3}}{|B|}^{\frac{4}{3}}{|X|}^{\frac{2}{3}}.\]
\end{proof}
\begin{lemma}\label{lemma3.2.7}

Let $A,B$ be finite subsets of $\mathbb{F}_p$ with $|A|\ll|B|\ll p^{\frac{2}{5}}$. Then, 
\[R(\mathbb{F}_p,A,B)\lesssim {|A|}^{\frac{11}{3}}{|B|}^3.\]
\end{lemma}
\begin{proof}
Let $Q:=(A+B)/(A+B)$ and $Z_t:=\{z\in\mathbb{F}_p\mid r_Q(z)\geq t\}$.
Note that for any $t\geq1$,
\[t|Z_t|\leq\sum_{z\in Z_t}r_Q(z)\leq\sum_{z\in Q}r_Q(z)={|A|}^2{|B|}^2.\]
\begin{enumerate}
\item Assume $|Z_t|\geq|A||B|$. Then,
\[|A|B|\geq t\Rightarrow|Z_t|\leq\frac{{|Z_t|}^2}{t}\leq\frac{{|A|}^3{|B|}^3}{t^2}.\]
\item On the other hand, suppose $|Z_t|\leq|A||B|$. Observe that every solution of the equation 
\[z=\frac{a_2+b_2}{a_1+b_1}\]
is a solution to
\[b_2-za_1=zb_1-a_2=y\in\mathbb{F}_p.\]
Thus, 
\begin{equation*}
\begin{split}
r_Q(z)&\leq\sum_{y\in\mathbb{F}_p}r_{B-zA}(y)r_{zB-A}(y)\leq\sum_{y\in\mathbb{F}_p}\frac{r_{B-zA}^2(y)+r_{zB-A}^2(y)}{2}\\
&=\frac{1}{2}\left(E^+(B,zA)+E^+(zB,A)\right).
\end{split}
\end{equation*}
We sum over $z\in Z_t$ and get
\begin{equation*}
\begin{split}
&t|Z_t|\leq\frac{1}{2}\sum_{z\in Z_t}\left(E^+(B,zA)+E^+(zB,A)\right)\ll {|A|}^{\frac{5}{3}}{|B|}^{\frac{4}{3}}{|Z_t|}^{\frac{2}{3}}\\
&\Rightarrow|Z_t|\ll\frac{{|A|}^5{|B|}^4}{t^3}
\end{split}
\end{equation*}
by lemma \ref{lemma3.2.6}.
\end{enumerate}
Therefore, 
\begin{equation*}
\begin{split}
R(\mathbb{F}_p,A,B)&=\sum_{z\in Q}r_Q^2(z)=\sum_{t=1}^{|Q|}t^2\left(|Z_t|-|Z_{t+1}|\right)\sim\sum_{t=1}^{|Q|}t|Z_t|\\
&=\sum_{t:|Z_t|\geq|A||B|}t|Z_t|+\sum_{t:|Z_t|\leq|A||B|}t|Z_t|\\
&\ll\sum_{t:|Z_t|\geq|A||B|}\frac{{|A|}^3{|B|}^3}{t}+\sum_{\substack{t:|Z_t|<|A||B|\\t\leq {|A|}^{\frac{4}{3}}|B|}}t|Z_t|+\sum_{\substack{t:|Z_t|<|A||B|\\t\geq {|A|}^{\frac{4}{3}}|B|}}\frac{{|A|}^5{|B|}^4}{t^2}\\
&\lesssim {|A|}^3{|B|}^3+{|A|}^{\frac{11}{3}}{|B|}^3= {|A|}^{\frac{11}{3}}{|B|}^3.
\end{split}
\end{equation*}
\end{proof}
\begin{lemma}\label{lemma3.2.8}
Let $A\subseteq\mathbb{F}_p$ with $|A|\ll p^{\frac{1}{5}}$. Then, there is a subset $A'\subseteq A$ such that 
\[|A'|\gtrsim {E_3^+(A)}^{\frac{1}{2}}{|A|}^{-1}\text{ and }{E_3^+(A)}^4{E^\times(A')}^3\ll {|A|}^{12}{|A'|}^{12}.\]
\end{lemma}
\begin{proof}
Via a regular dyadic pigeonhole argument, there is a number $t\leq|A|$ and a set of popular differences
\[P:=\{x\in A-A\mid t\leq r_{A-A}(x)<2t\}\text{ such that }E_3^+(A)\approx|P|t^3.\]
Again, by applying dyadic argument, there is a $q_1\leq|A|$ and a set of popular abscissae
\[A_1=\{a\in A\mid q_1\leq r_{P+A}(a)<2q_1\}\]
such that 
\[|A_1|q_1\approx\sum_{a\in A_1}r_{P+A}(a)\approx\sum_{x\in P}r_{A-A}(x)\approx|P|t.\]
Using dyadic argument once again, there is a $q_2\leq|A_1|$ and a set of popular coordinates
\[A_2=\{b\in A\mid q_2\leq r_{A_1-P}(b)<2q_2\}\]
such that 
\[|A_2|q_2\approx\sum_{b\in A_2}r_{A_1-P}(b)\approx\sum_{a\in A_1}r_{P+A}(a)\approx|A_1|q_1.\]
Since $q_2\lesssim|A_1|$, either $q_1\lesssim|A_1|$ or $q_2\lesssim|A_1|\lesssim q_1\lesssim|A_2|$. Assume 
\[q_2\lesssim|A_1|\lesssim q_1\lesssim|A_2|\]
For the other case, the proof is similar. By construction, for any $b\in A_2$, $|A\cap(P+b)|\approx q_2$. Now there are $E^\times(A_2)$ quadruples $(b_1,b_2,b_3,b_4)\in {A_2}^4$ such that 
\[\frac{b_1}{b_2}=\frac{b_3}{b_4}.\]
Thus, for every such tuple $(b_1,b_2,b_3,b_4)$, there are approximately ${q_2}^4$ choices of $(a_1,a_2,a_3,a_4)$ such that $a_i\in A_1\cap(P+b_i)$ and 
\[\frac{a_1-(a_1-b_1)}{a_2-(a_2-b_2)}=\frac{a_3-(a_3-b_3)}{a_4-(a_4-b_4)}.\]
Denote $s_i=a_i-b_i\in P$. Then,
\begin{equation}\label{eq19}
E^\times(A_2){\left(\frac{|P|t}{|A_2|}\right)}^4\approx E^\times(A_2){q_2}^4\ll N,
\end{equation}
where 
\begin{equation*}
\begin{split}
N&=\left|\left\{(a_1,\dotsc,a_4,s_1,\dotsc,s_4)\in A^4\times P^4\,\middle|\,\frac{a_1-s_1}{a_2-s_2}=\frac{a_3-s_3}{a_4-s_4}\in A_2/A_2\right\}\right|\\
&=R(Z,A,-P)
\end{split}
\end{equation*}
with $Z:=A_2/A_2$ and $-P:=\{-s\mid s\in P\}$. Denote 
\[r(z)=|\{(a_1,a_2,s_1,s_2)\in A^2\times P^2\mid (a_1-s_1)=z(a_2-s_2)\}|\Rightarrow N\leq\sum_{z\in Z}r^2(z).\]
Divide it into two cases:
\begin{enumerate}
\item Assume $|P|\leq|A|$. By lemma \ref{lemma3.2.7} and equation \eqref{eq19}, 
\begin{equation*}
\begin{split}
&N=R(Z,A,-P)\ll {|A|}^3{|P|}^{\frac{11}{3}}\Rightarrow E^\times(A_2){\left(\frac{|P|t}{|A_2|}\right)}^4\ll {|A|}^3{|P|}^{\frac{11}{3}}\\
&\Rightarrow {|P|}^{\frac{1}{3}}t^4\ll\frac{{|A|}^3{|A_2|}^4}{E^\times(A_2)}.
\end{split}
\end{equation*}
Thus, 
\begin{equation}\label{eq20}
E_3^+(A)\approx|P|t^3\ll {|P|}^{\frac{3}{4}}{\left(\frac{{|A|}^3{|A_2|}^4}{E^\times(A_2)}\right)}^{\frac{3}{4}}\ll\frac{{|A|}^3{|A_2|}^3}{{E^\times(A_2)}^{\frac{3}{4}}}.
\end{equation}
\item On the other hand, for $|P|>|A|$, it satisfies the assumption of Lemma \ref{Dung}. Thus, by Lemma \ref{Dung} and Equation \eqref{eq19},
\begin{equation*}
\begin{split}
&N=R(Z,A,-P)\lesssim {|A|}^{\frac{33}{8}}{|A_2|}^{\frac{5}{2}}\Rightarrow E^\times(A_2){\left(\frac{|P|t}{|A_2|}\right)}^4\lesssim {|A|}^{\frac{33}{8}}{|P|}^{\frac{5}{2}}\\
&\Rightarrow {|P|}^{\frac{3}{2}}t^4\lesssim\frac{{|A|}^{\frac{33}{8}}{|A_2|}^4}{E^\times(A_2)}.
\end{split}
\end{equation*}
Thus, 
\begin{equation}\label{eq21}
E_3^+(A)\approx|P|t^3\lesssim {|P|}^{-\frac{1}{8}}{\left(\frac{{|A|}^{\frac{33}{8}}{|A_2|}^4}{E^\times(A_2)}\right)}^{\frac{3}{4}}\ll\frac{{|A|}^{\frac{95}{32}}{|A_2|}^3}{{E^\times(A_2)}^{\frac{3}{4}}}
\end{equation}
\end{enumerate}
Combining equations \eqref{eq20} and \eqref{eq21}, we get
\[E_3^+(A)\lesssim\frac{{|A|}^3{|A_2|}^3}{{E^\times(A_2)}^{\frac{3}{4}}}\Rightarrow {E_3^+(A)}^4{E^\times(A_2)}^3\lesssim {|A|}^{12}{|A_2|}^{12}.\]
Additionally, since $t\leq|A|$, 
\[{|A_2|}^2\gtrsim|A_2|q_2\approx|P|t\approx\frac{E_3^+(A)}{t^2}\geq\frac{E_3^+(A)}{{|A|}^2}.\]
Set $A'=A_2$ and we get the conclusion.
\end{proof}
\begin{proof}[Proof of Proposition \ref{third_additive_energy_estimate}]
Similar to the proof of lemma \ref{lemma3.2.8}, there is a number $t\leq|A|$ and a set of popular differences
\[P:=\{x\in A-B\mid t\leq r_{A-B}(x)<2t\}\text{ such that }E_3^+(A,B)\approx|P|t^3.\]
By dyadic decomposition again, there is a number $q\leq|A|$ and a set of popular abscissae 
\[A_1:=\{a\in A\mid q\leq r_{P+B}(a)<2q\}\]
such that
\[|A_1|q\approx\sum_{a\in P+B}r_{P+B}(a)\approx\sum_{x\in P}r_{A-B}(x)\approx|P|t.\]
Divide it into two cases:
\begin{enumerate}
\item If $|P|\leq|B|$, then 
\begin{equation}\label{eq22}
\begin{split}
&\frac{{|A_1|}^4}{|A_1\cdot A_1|}\times\frac{{|P|}^4t^4}{{|A_1|}^4}\lesssim E^\times(A_1)q^4\approx R(\mathbb{F}_p,P,B)\lesssim {|P|}^{\frac{11}{3}}{|B|}^3\\
&\Rightarrow {|P|}^{\frac{1}{3}}t^4\lesssim|A_1\cdot A_1|{|B|}^3.\\
&\Rightarrow E_3^+(A,B)\approx {|P|}^{\frac{3}{4}}{\left({|P|}^{\frac{1}{3}}t^4\right)}^{\frac{3}{4}}\lesssim {|P|}^{\frac{3}{4}}{|A_1\cdot A_1|}^{\frac{3}{4}}{|B|}^{\frac{9}{4}}\lesssim {|A_1\cdot A_1|}^{\frac{3}{4}}{|B|}^3.
\end{split}
\end{equation}
\item If $|P|>|B|$, then since $|A_1/A_1|\lesssim {|B|}^2$, we can apply Lemma \ref{Dung} to get
\[\frac{{|A_1|}^4}{|A_1\cdot A_1|}\times\frac{{|P|}^4t^4}{{|A_1|}^4}\lesssim E^\times(A_1)q^4\approx R(\mathbb{F}_p,P,B)\lesssim {|B|}^{\frac{33}{8}}{|P|}^{\frac{5}{2}}.\]
That is, 
\begin{equation}\label{eq23}
\begin{split}
&{|P|}^{\frac{3}{2}}t^4\lesssim|A\cdot A|{|B|}^{\frac{33}{8}}\Rightarrow {\left(|P|t^3\right)}^{\frac{4}{3}}\lesssim {|P|}^{-\frac{1}{6}}|A\cdot A|{|B|}^{\frac{33}{8}}\\
&\Rightarrow {E_3^+(A,B)}^{\frac{4}{3}}\approx {\left(|P|t^3\right)}^{\frac{4}{3}}\lesssim|A\cdot A|{|B|}^{\frac{95}{24}}\\
&\Rightarrow E_3^+(A,B)\lesssim {|A\cdot A|}^{\frac{3}{4}}{|B|}^{\frac{95}{32}}.
\end{split}
\end{equation}
\end{enumerate}
Combining equations \eqref{eq22} and \eqref{eq23}, we get
\[E_3^+(A,B)\lesssim {|A\cdot A|}^{\frac{3}{4}}{|B|}^3\Rightarrow {E_3^+(A,B)}^{\frac{4}{3}}{|B|}^{-4}\lesssim|A\cdot A|,\]
which completes this proof.
\end{proof}
Finally, using the above lemmas, we obtain a new sum-product estimate in $\mathbb{F}_p$.
\begin{theorem}\label{my_first_result}
Let $A\subseteq\mathbb{F}_p$ with $|A|\ll p^{\frac{2}{5}}$. Then, 
\[{|A+A|}^8{|A\cdot A|}^3\gtrsim {|A|}^{12}.\]
Especially, we have
\[\max\{|A+A|,|A\cdot A|\}\gtrsim {|A|}^{1+\frac{1}{11}}.\]
\end{theorem}
\begin{proof}
By H\"{o}lder inequality,
\[{|A|}^2=\sum_{x\in A+A}r_{A+A}(x)\leq {|A+A|}^{\frac{2}{3}}{\left(\sum_{x\in A+A}r_{A+A}^3(x)\right)}^{\frac{1}{3}}\Rightarrow |A+A|\geq\frac{{|A|}^3}{{E_3^+(A)}^{\frac{1}{2}}}.\]
Together with Proposition \ref{third_additive_energy_estimate}, we have
\begin{equation*}
\begin{split}
&\left\{\begin{array}{l}
{E_3^+(A)}^{\frac{4}{3}}{|A|}^{-4}\lesssim|A\cdot A|\\
{E_3^+(A)}^{-\frac{1}{2}}{|A|}^3\leq|A+A|
\end{array}\right.\Rightarrow\left\{\begin{array}{l}
{E_3^+(A)}^4{|A|}^{-12}\lesssim {|A\cdot A|}^3\\
{E_3^+(A)}^{-4}{|A|}^{24}\leq {|A+A|}^8
\end{array}\right.\\
&\Rightarrow {|A|}^{12}\lesssim {|A+A|}^8{|A\cdot A|}^3,
\end{split}
\end{equation*}
which completes the proof.
\end{proof}
\begin{remark}
Although we can see that if we compare Theorem \ref{my_first_result} and Theorem \ref{Stevens}, the latter is stronger in general, the proving technique of Theorem \ref{my_first_result} is totally different from Elekes' method, so this provides a new perspective to Erd\"{o}s-Szemer\'{e}di Conjecture in $\mathbb{F}_p$.
\end{remark}
\subsection{A better estimate of additive energy in \texorpdfstring{$\mathbb{F}_p$}{}}
In this part, we use ``the point-plane incidence" bound, proved in \cite{Rudnev2018number}, to give another estimate. This is also a piece of evidence to show the strong connection between Incidence Geometry and sum-product estimate.
\begin{theorem}[Theorem 3, \cite{Rudnev2018number}]\label{theorem3.3.1}
Let $\mathcal{P},\Pi$ be sets of points and planes, of cardinalities respectively $m$ and $n$ in $\mathbb{F}_p$. Suppose that $m\geq n$ and $n=O(p^2)$. Let $k$ be the maximum number of collinear planes. Then, 
\[|\mathcal{I}(\mathcal{P},\Pi)|=O(m\sqrt{n}+km).\]
\end{theorem}
Using this bound, we prove the following theorem.
\begin{theorem}\label{theorem3.3.2}
Let $A\subseteq\mathbb{F}_p$ with $|A|\ll p^{\frac{1}{2}}$. Then, 
\[{|A+A|}^2{|A\cdot A|}^3\gtrsim {|A|}^6.\]
Especially, we have
\[\max\{|A+A|,|A\cdot A|\}\gtrsim {|A|}^{1+\frac{1}{5}}.\]
\end{theorem}
\begin{proof}
Notice that 
\begin{equation}\label{eq24}
\begin{split}
E^+(A)&=\left|\left\{(a_1,a_2,a_3,a_4)\in A^4\,\middle|\,a_1+a_2=a_3+s_4\right\}\right|\\
&={|A|}^{-2}N_1\leq {|A|}^{-2}N_2
\end{split}
\end{equation}
where 
\[N_1:=\left|\left\{(a_1,a_2,a_3,a_4,a_5,a_6)\in A^6\,\middle|\,(a_1a_2){a_2}^{-1}+a_5=(a_3a_4){a_4}^{-1}+a_6\right\}\right|\]
and $N_2$ is defined as 
\[\left|\left\{(u_1,u_2,v_1,v_2,a_1,a_2)\in {(A\cdot A)}^2\times {\left(A^{-1}\right)}^2\times A^2\,\middle|\,u_1v_1+a_1=u_2v_2+a_2\right\}\right|.\]
Consider
\[\mathcal{P}:=\{(u_1,a_1,v_2)\in(A\cdot A)\times A\times A^{-1}\}\]
and
\[\Pi:=\{\pi:ax+y=bz+c\mid a\in A^{-1},b\in A\cdot A,c\in A\}.\] 
Note that $|\mathcal{P}|=|\Pi|=|A\cdot A|{|A|}^2\ll p^2$ and 
\[\mathcal{I}(\mathcal{P},\Pi):=|\{(p,\pi)\in\mathcal{P}\times\Pi\mid p\in\pi\}|=N_2.\] 
\par
For any line
\[l=\{(x,y,z)=(x_0+\alpha t,y_0+\beta t,z_0+\gamma t)\mid t\in\mathbb{F}_p\},\]
define
\[\Pi_l:=\{\pi\in\Pi\mid l\in\pi\}.\]
Then, $k=\max_l|\Pi_l|$. For a plane $\pi:ax+y=bz+c\in\Pi_l$, we have
\[\left\{\begin{array}{l}
ax_0+y_0=bz_0+c\\
a\alpha+\beta=b\gamma
\end{array}\right.\Rightarrow\left\{\begin{array}{l}
b=b(a)\\
c=c(a)
\end{array}\right.\text{ if such }b,c\text{ exist.}\]
That is, $|\Pi_l|\leq|A|\Rightarrow k\leq|A|$. Therefore, by equation \eqref{eq24} and Theorem \ref{theorem3.3.1}, we know that 
\begin{equation*}
\begin{split}
E^+(A)&\leq {|A|}^{-2}N_2={|A|}^{-2}\mathcal{I}(\mathcal{P},\Pi)\\
&\ll {|A|}^{-2}\left({|A\cdot A|}^{\frac{3}{2}}{|A|}^3+|A||A\cdot A|{|A|}^2\right)\ll |A|{|A\cdot A|}^{\frac{3}{2}}.
\end{split}
\end{equation*}
Applying Cauchy-Schwarz inequality, we have
\[\frac{{|A|}^4}{|A+A|}\ll |A|{|A\cdot A|}^{\frac{3}{2}}\Rightarrow {|A|}^6\ll {|A+A|}^2{|A\cdot A|}^3.\]
\end{proof}
\begin{remark}
In this theorem, we can see that the exponential part, $1+\frac{1}{5}$ is as same as Corollary \ref{Stevens}. Therefore, in this sense, we obtained a result as good as Stevens did. However, in Corollary \ref{Stevens2}, we can see that in extremal cases, Corollary \ref{Stevens2} is still better than our result. In spite of this, the idea of adapting point-plane incidence to proofs of sum-product theorems is still valueless.
\end{remark}
\section{Other Incidence Geometry Problems}
In the previous sections, we already realize how powerful an Incidence Geometry result can be in the sum-product estimate. Therefore, this section will focus on several kinds of Incidence Geometry problems whose connection to sum-product estimate is more implicit than previous ones and briefly introduce their relation to sum-product estimate.
\subsection{Collinear lines in grids}
In the proof of Theorem 6.2 in \cite{Bourgain2004sum}, Bourgain, Katz, and Tao observed that if the point-line incidence $\mathcal{I}(\mathcal{P},\mathcal{L})$ is large, then the point set should contain several large grid-like structures. Hence, to improve the upper bound of $\mathcal{I}(\mathcal{P},\mathcal{L})$, a more delicate result about the incidence relation between lines and grids must be needed. Additionally, Xue (2020, \cite{Xue2021asymmetric}) already applied this idea in the real number, and this section will extend his skill to the finite field $\mathbb{F}_p$. To start with, we define an incidence amount as follows.
\begin{definition}[$T^o,T$]
For any three subsets $A,B,C\subseteq\mathbb{F}_p$, define 
\begin{equation*}
\begin{split}
&T^o(A,B,C)\\
&:=\left|\{(u_1,u_2,u_3)\in A^2\times B^2\times C^2\mid \{u_1,u_2,u_3\}\text{ are collinear and distinct}\}\right|
\end{split}
\end{equation*}
and 
\begin{equation*}
\begin{split}
T(A,B,C):=&\text{ the number of tuples }(a_1,a_2,b_1,b_2,c_1,c_2)\in A^2\times B^2\times C^2\\
&\text{ with }(b_1-a_1)(c_2-a_2)=(c_1-a_1)(b_2-a_2).
\end{split}
\end{equation*}
\end{definition}
Now, we want to bound $T(A, B, C)$, but before that, we need the following Lemma \ref{lemma4.1.2} and \ref{lemma4.1.3} to obtain an upper bound for $T^o(A, B, C)$ and use the relation between $T$ and $T^o$ to get the desired upper bound in Proposition \ref{proposition4.1.4}.
\begin{lemma}\label{lemma4.1.2}
Let $A,B$ be subsets of $\mathbb{F}_p$ with $|A|\leq|B|\ll p^{\frac{2}{5}}$ and $\mathcal{P}=A\times B$. Then,
\[|\mathcal{L}|\ll\min\left\{\frac{{|A|}^3{|B|}^2}{k^4},\frac{{|A|}^2{|B|}^2}{k^2}\right\}+\frac{|\mathcal{L}|}{k}\]
with 
\[\mathcal{L}:=\{l:\text{a line in }{\mathbb{F}_p}^2\mid |l\cap(A\times B)|\geq k\}.\]
\end{lemma}
\begin{proof}
Let $\mathcal{L}$ be the collection of lines $l$ with $|l\cap(A\times B)|\geq k$. Thus, 
\[k|\mathcal{L}|\leq\mathcal{I}(A\times B,\mathcal{L}).\]
If ${|\mathcal{L}|}^3<|A|{|B|}^2$, this lemma holds. i.e. We may assume ${|\mathcal{L}|}^3\geq|A|{|B|}^2$. Additionally, 
\[|A||\mathcal{L}|\ll |A|{|A\times B|}^2\ll p^2,\]
so we can apply theorem \ref{Szemeredi-Trotter_in_Fp}. By theorem \ref{Szemeredi-Trotter_in_Fp}, 
\begin{equation}\label{eq25}
\mathcal{I}(A\times B,\mathcal{L})\ll {|A|}^{\frac{3}{4}}{|B|}^{\frac{1}{2}}{|\mathcal{L}|}^{\frac{3}{4}}+|\mathcal{L}|\Rightarrow|\mathcal{L}|\ll\frac{{|A|}^3{|B|}^2}{k^4}+\frac{|\mathcal{L}|}{k}.
\end{equation}
Moreover, by Cauchy-Schwarz incidence bound, 
\begin{equation}\label{eq26}
k|\mathcal{L}|\leq\mathcal{I}(A\times B,\mathcal{L})\ll |A||B|{|\mathcal{L}|}^{\frac{1}{2}}+|\mathcal{L}|\Rightarrow|\mathcal{L}|\ll\frac{{|A|}^2{|B|}^2}{k^2}+\frac{|\mathcal{L}|}{k}.
\end{equation}
We complete the proof by combining equations \eqref{eq25} and \eqref{eq26}.
\end{proof}
\begin{lemma}\label{lemma4.1.3}
Suppose that $A_1,A_2,A_3\subset\mathbb{F}_p$ with $|A_1|\leq|A_2|\leq|A_3|\ll p^{\frac{2}{5}}$. Let $\mathcal{L}$ be the collection of lines such that $l\in\mathcal{L}$ if and only if $l$ contains three distinct points $u_1,u_2,u_3$ with $u_i\in {A_i}^2$. For each $l\in\mathcal{L}$ and $i=1,2,3$, define $\alpha_{i,l}=|l\cap(A_i\times A_i)|$. For $1\leq k\leq|A_i|$, define $\mathcal{L}_{i,k}:=\{l\in\mathcal{L}\mid\alpha_{i,l}\geq k\}$. Then, we have
\[\sum_{l\in\mathcal{L}_{i,2}}{\alpha_{i,l}}^s\lesssim {|A_i|}^4+{|A_i|}^{\frac{s}{2}+3}+{|A_i|}^{s+1}\text{ for }i=1,2,3\text{ and }s>1.\]
\end{lemma}
\begin{proof}
Note that 
\begin{equation*}
\begin{split}
\sum_{l\in\mathcal{L}_{i,2}}{\alpha_{i,l}}^s&=\sum_{2\leq k\leq|A_i|}|\{l\in\mathcal{L}_{i,2}\mid\alpha_{i,l}=k\}|k^s\\
&\approx\sum_{2\leq k\leq|A_i|}\left(|\{l\in\mathcal{L}_{i,2}\mid\alpha_{i,l}=k\}|\sum_{2\leq m\leq k}m^{s-1}\right)\\
&=\sum_{2\leq m\leq|A_i|}m^{s-1}\underbrace{\sum_{k\geq m,k\leq|A_i|}|\{l\in\mathcal{L}_{i,2}\mid\alpha_{i,l}=k\}|}_{=|\{l\in\mathcal{L}_{i,2}\mid\alpha_{i,l}\geq m\}|}\\
&\approx\sum_{2\leq m\leq|A_i|}m^{s-1}|\{l\in\mathcal{L}_{i,2}\mid\alpha_{i,l}\geq m\}|=\sum_{2\leq m\leq|A_i|}m^{s-1}|\mathcal{L}_{i,m}|.
\end{split}
\end{equation*}
\par
Divide it into three parts:
\begin{enumerate}
\item For $2\leq m\leq\log|A_i|$, via lemma \ref{lemma4.1.2}, 
\[|\mathcal{L}_{i,m}|\lesssim\frac{{|A_i|}^4}{m^2}+\frac{|\mathcal{L}|}{m}\lesssim m^{-1}{|A_i|}^4.\]
\item Similarly, for $\log|A_i|\leq m\leq {|A_i|}^{\frac{1}{2}}$, we can get
\[|\mathcal{L}_{i,m}|\lesssim m^{-2}{|A_i|}^4.\]
\item As for ${|A_i|}^{\frac{1}{2}}\leq m\leq|A_i|$, we obtain
\[|\mathcal{L}_{i,m}|\lesssim\frac{{|A_i|}^5}{m^4}+\frac{|\mathcal{L}|}{m}\lesssim m^{-4}{|A_i|}^5.\]
\end{enumerate}
To sum up, we have
\begin{equation*}
\begin{split}
\sum_{l\in\mathcal{L}_{i,2}}{\alpha_{i,l}}^s\lesssim&\sum_{2\leq m\leq\log|A_i|}m^{s-2}{|A_i|}^4+\sum_{\log|A_i|\leq m\leq {|A_i|}^{\frac{1}{2}}}m^{s-3}{|A_i|}^4\\
&+\sum_{{|A_i|}^{\frac{1}{2}}\leq m\leq|A_i|}m^{s-5}{|A_i|}^5\\
\lesssim& {|A_i|}^4+\left(1+{|A_i|}^{\frac{s-2}{2}}\right){|A_i|}^4+\left({|A_i|}^{\frac{s-4}{2}}+{|A_i|}^{s-4}\right){|A_i|}^5\\
\lesssim& {|A_i|}^4+{|A_i|}^{\frac{s}{2}+3}+{|A_i|}^{s+1}.
\end{split}
\end{equation*}
\end{proof}
\begin{proposition}\label{proposition4.1.4}
Suppose that $A_1,A_2,A_3$ are three finite subsets of $\mathbb{F}_p$ with $|A_1|\leq|A_2|\leq|A_3|\ll p^{\frac{2}{5}}$. Then, 
\begin{equation*}
\begin{split}
T^o(A_1,A_2,A_3)\lesssim& \min\left\{|A_1|{|A_2|}^{\frac{9}{4}}{|A_3|}^{\frac{5}{4}},{|A_1|}^{\frac{1}{2}}{|A_2|}^{\frac{5}{2}}{|A_3|}^{\frac{5}{4}},{|A_2|}^2{|A_3|}^2\right\}\\
&+{|A_1|}^2{|A_2|}^{\frac{5}{4}}{|A_3|}^{\frac{5}{4}}+{|A_1|}^{\frac{3}{2}}{|A_2|}^{\frac{3}{2}}{|A_3|}^{\frac{5}{4}}
\end{split}
\end{equation*}
and
\[T(A_1,A_2,A_3)\lesssim T^o(A_1,A_2,A_3)+{|A_1|}^2{|A_3|}^2.\]
\end{proposition}
\begin{proof}
It is clear that 
\[T^o(A_1,A_2,A_3)\leq\sum_{l\in\mathcal{L}}\alpha_{1,l}\alpha_{2,l}\alpha_{3,l}.\]
Denote $I_1=\{1\}$ and $I_2=\mathbb{N}\setminus\{1\}$ and for $(j_1,j_2,j_3)\in {\{1,2\}}^3$, define
\[S_{j_1,j_2,j_3}:=\sum_{\substack{l\in\mathcal{L}\\\alpha_{i,l}\in I_{j_i}}}\alpha_{1,l}\alpha_{2,l}\alpha_{3,l}.\]
By Cauchy-Schwarz inequality and Lemma \ref{lemma4.1.3}, we have 
\begin{equation*}
\begin{split}
S_{1,2,2}&\leq\sum_{l\in\left(\mathcal{L}_{2,2}\cap\mathcal{L}_{3,2}\right)\setminus\mathcal{L}_{1,2}}\alpha_{2,l}\alpha_{3,l}\leq {|\mathcal{L}|}^{\frac{1}{2}}{\left(\sum_{l\in\mathcal{L}_{2,2}\cap\mathcal{L}_{3,2}}{\left(\alpha_{2,l}\alpha_{3,l}\right)}^2\right)}^{\frac{1}{2}}\\
&\ll {\left({|A_1|}^2{|A_2|}^2\right)}^{\frac{1}{2}}{\left({\left(\sum_{l\in\mathcal{L}_{2,2}}{\left({\alpha_{2,l}}^2\right)}^2\right)}^{\frac{1}{2}}{\left(\sum_{l\in\mathcal{L}_{3,2}}{\left({\alpha_{3,l}}^2\right)}^2\right)}^{\frac{1}{2}}\right)}^{\frac{1}{2}}\\
&\ll |A_1|{|A_2|}^{\frac{9}{4}}{|A_3|}^{\frac{5}{4}}.
\end{split}
\end{equation*}
Also, by Lemma \ref{lemma4.1.3},
\begin{equation*}
\begin{split}
S_{1,2,2}&\leq\sum_{l\in\mathcal{L}_{2,2}\cap\mathcal{L}_{3,2}}\alpha_{2,l}\alpha_{3,l}\leq {\left(\sum_{l\in\mathcal{L}_{2,2}}{\alpha_{2,l}}^2\right)}^{\frac{1}{2}}{\left(\sum_{l\in\mathcal{L}_{3,2}}{\alpha_{3,l}}^2\right)}^{\frac{1}{2}}\\
&\leq {\left(\sum_{l\in\mathcal{L}_{2,2}}{\alpha_{2,l}}^2\right)}^{\frac{1}{2}}{\left({|\mathcal{L}_{3,2}|}^{\frac{1}{2}}{\left(\sum_{l\in\mathcal{L}_{3,2}}{\left({\alpha_{3,l}}^2\right)}^2\right)}^{\frac{1}{2}}\right)}^{\frac{1}{2}}\\
&\lesssim {|A_2|}^2{\left({|A_1|}^2{|A_2|}^2\right)}^{\frac{1}{4}}{|A_3|}^{\frac{5}{4}}\lesssim {|A_1|}^{\frac{1}{2}}{|A_2|}^{\frac{5}{2}}{|A_3|}^{\frac{5}{4}}.
\end{split}
\end{equation*}
Moreover, 
\begin{equation*}
S_{1,2,2}\leq\sum_{l\in\mathcal{L}_{2,2}\cap\mathcal{L}_{3,2}}\alpha_{2,l}\alpha_{3,l}\leq {\left(\sum_{l\in\mathcal{L}_{2,2}}{\alpha_{2,l}}^2\right)}^{\frac{1}{2}}{\left(\sum_{l\in\mathcal{L}_{3,2}}{\alpha_{3,l}}^2\right)}^{\frac{1}{2}}\ll {|A_2|}^2{|A_3|}^2.
\end{equation*}
Thus, combining the above three bounds, we get
\begin{equation}\label{eq27}
S_{1,2,2}\ll\min\left\{|A_1|{|A_2|}^{\frac{9}{4}}{|A_3|}^{\frac{5}{4}},{|A_1|}^{\frac{1}{2}}{|A_2|}^{\frac{5}{2}}{|A_3|}^{\frac{5}{4}},{|A_2|}^2{|A_3|}^2\right\}.
\end{equation}
Note that for $i=2,3$, 
\[\sum_{\substack{l\in\mathcal{L}\\\alpha_{i,l}=1}}{\alpha_{i,l}}^s\leq|\mathcal{L}|\leq {|A_1|}^2{|A_2|}^2\leq {|A_i|}^4.\]
Thus, by Cauchy-Schwarz inequality and Lemma \ref{lemma4.1.3}, we have
\begin{equation}\label{eq28}
\begin{split}
S_{2,j_2,j_3}\leq&{\left(\sum_{l\in\mathcal{L}_{1,2}}{\alpha_{1,l}}^2\right)}^{\frac{1}{2}}{\left(\sum_{\substack{l\in\mathcal{L}\\\alpha_{2,l}\in I_{j_2},\alpha_{3,l}\in I_{j_3}}}{\left(\alpha_{2,l}\alpha_{3,l}\right)}^2\right)}^{\frac{1}{2}}\\
\leq&{\left(\sum_{l\in\mathcal{L}_{1,2}}{\alpha_{1,l}}^2\right)}^{\frac{1}{2}}{\left({\left(\sum_{\substack{l\in\mathcal{L}\\\alpha_{2,l}\in I_{j_2}}}{\left({\alpha_{2,l}}^2\right)}^2\right)}^{\frac{1}{2}}{\left(\sum_{\substack{l\in\mathcal{L}\\\alpha_{3,l}\in I_{j_3}}}{\left({\alpha_{3,l}}^2\right)}^2\right)}^{\frac{1}{2}}\right)}^{\frac{1}{2}}\\
=&{\left(\sum_{l\in\mathcal{L}_{1,2}}{\alpha_{1,l}}^2\right)}^{\frac{1}{2}}{\left(\sum_{\substack{l\in\mathcal{L}\\\alpha_{2,l}\in I_{j_2}}}{\alpha_{2,l}}^4\right)}^{\frac{1}{4}}{\left(\sum_{\substack{l\in\mathcal{L}\\\alpha_{3,l}\in I_{j_3}}}{\alpha_{3,l}}^4\right)}^{\frac{1}{4}}\\
\ll&{|A_1|}^2{|A_2|}^{\frac{5}{4}}{|A_3|}^{\frac{5}{4}}
\end{split}
\end{equation}
Similarly, for $S_{1,1,j_3}$ and $S_{1,2,1}$, we have the following estimates:
\begin{equation}\label{eq29}
\begin{split}
S_{1,1,j_3}&\leq\sum_{\substack{l\in\mathcal{L}\\\alpha_{3,l}\in I_{j_3}}}\alpha_{3,l}\leq {|\mathcal{L}|}^{\frac{3}{4}}{\left(\sum_{\substack{l\in\mathcal{L}\\\alpha_{3,l}\in I_{j_3}}}{\alpha_{3,l}}^4\right)}^{\frac{1}{4}}\\
&\lesssim {\left({|A_1|}^2{|A_2|}^2\right)}^{\frac{3}{4}}{\left({|A_3|}^5\right)}^{\frac{1}{4}}={|A_1|}^{\frac{3}{2}}{|A_2|}^{\frac{3}{2}}{|A_3|}^{\frac{5}{4}}
\end{split}
\end{equation}
and 
\begin{equation}\label{eq30}
\begin{split}
S_{1,2,1}&\leq\sum_{l\in\mathcal{L}_{2,2}}\alpha_{2,l}\leq {|\mathcal{L}|}^{\frac{3}{4}}{\left(\sum_{\substack{l\in\mathcal{L}\\\alpha_{2,l}\in I_{j_2}}}{\alpha_{2,l}}^4\right)}^{\frac{1}{4}}\lesssim {\left({|A_1|}^2{|A_2|}^2\right)}^{\frac{3}{4}}{\left({|A_2|}^5\right)}^{\frac{1}{4}}\\
&\ll {|A_1|}^{\frac{3}{2}}{|A_2|}^{\frac{11}{4}}
\end{split}
\end{equation}
To sum up, by equation \eqref{eq27}, \eqref{eq28}, \eqref{eq29}, and \eqref{eq30}, we know that 
\begin{equation*}
\begin{split}
T^o(A_1,A_2,A_3)\leq&\sum_{j_1,j_2,j_3}S_{j_1,j_2,j_3}\\
\lesssim&\min\left\{|A_1|{|A_2|}^{\frac{9}{4}}{|A_3|}^{\frac{5}{4}},{|A_1|}^{\frac{1}{2}}{|A_2|}^{\frac{5}{2}}{|A_3|}^{\frac{5}{4}},{|A_2|}^2{|A_3|}^2\right\}\\
&+{|A_1|}^2{|A_2|}^{\frac{5}{4}}{|A_3|}^{\frac{5}{4}}+{|A_1|}^{\frac{3}{2}}{|A_2|}^{\frac{3}{2}}{|A_3|}^{\frac{5}{4}}+{|A_1|}^{\frac{3}{2}}{|A_2|}^{\frac{11}{4}}\\
\lesssim&\min\left\{|A_1|{|A_2|}^{\frac{9}{4}}{|A_3|}^{\frac{5}{4}},{|A_1|}^{\frac{1}{2}}{|A_2|}^{\frac{5}{2}}{|A_3|}^{\frac{5}{4}},{|A_2|}^2{|A_3|}^2\right\}\\
&+{|A_1|}^2{|A_2|}^{\frac{5}{4}}{|A_3|}^{\frac{5}{4}}+{|A_1|}^{\frac{3}{2}}{|A_2|}^{\frac{3}{2}}{|A_3|}^{\frac{5}{4}}.
\end{split}
\end{equation*}
As for the estimate of $T(A_1,A_2,A_3)$, consider those terms which are counted by $T(A_1,A_2,A_3)$ but not by $T^o(A_1,A_2,A_3)$. For $(a_1,a_2)\in {A_1}^2,(b_1,b_2)\in {A_2}^2,(c_1,c_2)\in {A_3}^2$ with 
\begin{equation}\label{eq31}
(b_1-a_1)(c_2-a_2)=(c_1-a_1)(b_2-a_2),
\end{equation}
if they are distinct and collinear, then $a_1\neq b_1,a_1\neq c_1$, and $b_1\neq c_1$. On the other hand, if $a_1=b_1$, there are at most 
\[|A_1\cap A_2\cap A_3||A_1||A_2||A_3|+{|A_1\cap A_2|}^2{|A_3|}^2\]
solutions to the equation \eqref{eq31}. For the case $a_1=c_1$ and $b_1=c_1$, we have the bound 
\[|A_1\cap A_2\cap A_3||A_1||A_2||A_3|+{|A_1\cap A_3|}^2{|A_2|}^2\]
and
\[|A_1\cap A_2\cap A_3||A_1||A_2||A_3|+{|A_2\cap A_3|}^2{|A_1|}^2\text{, respectively.}\]
To sum up, we have $T(A_1,A_2,A_3)-T^o(A_1,A_2,A_3)\ll{|A_1|}^2{|A_3|}^2$.
\end{proof}
\subsection{The number of bisectors in a subset of \texorpdfstring{${\left(\faktor{\mathbb{Z}}{p^3\mathbb{Z}}\right)}^2$}{}}
While working on problems about distances between pairs of points in a given point set $\mathcal{P}\subseteq\mathbb{R}^2$, it is natural to ask the total number of possible bisectors and the incidence relation between them. Moreover, Hanson, Lund, and Roche-Newton generalized this question in $\mathbb{F}_p^2$ in \cite{Hanson2016distinct}. To start with, for a vector $x=(x_1,x_2)\in\mathbb{F}_p^2$, define 
\[\|x\|\equiv {x_1}^2+{x_2}^2\mod p\]
and for two points $x,y\in\mathbb{F}_p^2$, call $\|x-y\|$ as the ``distance'' between them. Notice that although this distance is not a norm (even not a metric) in the mathematical sense, it still preserves several structures of $\mathbb{F}_p^2$. Next, the authors defined the ``perpendicular bisector'' between two points $x,y\in\mathbb{F}_p^2$ to be
\[B(x,y):=\{z\in\mathbb{F}_p^2\mid\|z-x\|=\|z-y\|\}\]
and for a point set $\mathcal{P}\subseteq\mathbb{F}_p^2$, define
\[B(\mathcal{P}):=\{B(x,y)\mid x,y\in\mathcal{P}\}.\]
After these settings, in \cite{Hanson2016distinct}, the authors proved the following theorem:
\begin{theorem}[Theorem 1, \cite{Hanson2016distinct}]\label{theorem4.2.1}
If a point set $\mathcal{P}\subseteq\mathbb{F}_p^2$ with $|\mathcal{P}|\gg p^{\frac{3}{2}}$, then 
\[|B(\mathcal{P})|\gg p^2.\]
In other words, the order of their perpendicular bisector is as large as the order of all lines.
\end{theorem}
\par
One of my research projects is to extend this result to more general settings. However, it turns out it is a difficult problem to obtain a similar result as in $\mathbb F_p$ because losing the structure of being a field creates many difficult problems. We now consider point sets in ${\left(\faktor{\mathbb{Z}}{p^3\mathbb{Z}}\right)}^2$ where $p$ is a Gaussian prime, instead of in $\mathbb{F}_p^2$. In the following text, we denote $\faktor{\mathbb{Z}}{p^3\mathbb{Z}}$ as $\mathbb{Z}_{p^3}$ or $\mathbb{Z}_q$, where $q=p^3$, for simplicity. The reason why we choose $\mathbb{Z}_q^2$ out of the general case ${\left(\faktor{\mathbb{Z}}{n\mathbb{Z}}\right)}^2$ is that we want to reduce the number of solutions to $\|x\|=0$, which would be very annoying if there are too many solutions. Notice that the above definition of $\|\cdot\|$, ``distance,'' and ``perpendicular bisector'' are also valid in the case of $\mathbb{Z}_q^2$. Our study in this scope is the following result in $\Z_q^2$ which is based on a conjecture that I made which I believe to be true.
\begin{theorem}\label{main}
Assume that Conjecture \ref{key_conj} holds. For any $m>\frac{5}{3}$ and point set $\mathcal{P}\subseteq\Z_q^2$ with $|\mathcal{P}|\gg q^m$, we have
\[|B(\mathcal{P})|\gg q^{2(m-1)}.\]
\end{theorem}
\par
To prove this theorem, let us start with several fundamental geometric definitions in $\mathbb{Z}_q^2$.
\begin{definition}[rotation, reflection, and translation]
~
\begin{enumerate}
\item A matrix of the form
\[\begin{pmatrix}a&-b\\b&a\end{pmatrix},a^2+b^2=1\]
is called a rotation matrix. For $u\in \mathbb{Z}_q^2$ and $R$ is a rotation matrix, then a \textit{rotation} about $u$ is an affine map of the form
\[\mathcal{R}:v\mapsto R(v-u)+u.\]
\item A matrix of the form 
\[\begin{pmatrix}a&b\\b&-a\end{pmatrix},a^2+b^2=1\]
is called a reflection matrix. A \textit{reflection} about $u$ by a reflection matrix $S$ is an affine map of the form 
\[\mathcal{S}:v\mapsto S(v-u)+u.\]
\item A \textit{translation} by $u$ is an affine map of the form
\[\mathcal{T}:v\mapsto v+u.\]
We call a translation non-trivial if $u\neq0$.
\end{enumerate}
\end{definition}
\begin{definition}[non-isotropic]
A line $l:ax+by=c$ is called \textit{non-isotropic} if $a,b$, and $a^2+b^2$ are units. Otherwise, it is isotropic.
\end{definition}
Based on these definitions above, we followed the steps of Hanson, Lund, and Roche-Newton and modified their proving techniques to generalize their results into the case of $\Z_q^2$. To prove the main theorem, we need the following Lemma \ref{lemma4.2.5} to Lemma \ref{lemma4.2.8}. However, because their proof will need more fundamental lemmas, we move their proofs to Appendix C. for the sake of fluency.
\begin{lemma}\label{lemma4.2.5}
For $d\in\Z_q$, let $Q(d)$ be the number of solutions to $x^2\equiv d\mod q$. 
\[Q(d)=\begin{cases}
2&\text{if }\legendre{d}{p}=1\\
0&\text{if }\legendre{d}{p}=-1\\
0&\text{if }p\mid d\text{ and }p^2\nmid d\\
2p&\text{if }p^2\mid d\text{ and }\legendre{d/p^2}{p}=1\\
0&\text{if }p^2\mid d\text{ and }\legendre{d/p^2}{p}=-1\\
p&\text{if }d=0
\end{cases}\]
where $\legendre{d}{p}$ is the Legendre symbol.
\end{lemma}
\begin{lemma}\label{lemma4.2.6}
If $x,y,z,w\in\mathbb{Z}_q^2$ are such that $B=B(x,z)=B(y,w)$ is non-isotropic then 
\[\|x-y\|=\|z-w\|.\]
\end{lemma}
\begin{lemma}\label{lemma4.2.7}
For a given point $u\in\mathbb{Z}_q^2$ with $q=p^3$, there are $p^3-p^2$ non-isotropic lines passing $u$.
\end{lemma}
\begin{lemma}\label{lemma4.2.8}
Suppose $u\in\mathbb{Z}_q^2$ and $\rho\in\mathbb{Z}_q$. Then, we have:
\[|C_\rho(u)|=\begin{cases}
p^2&\text{if }\rho\equiv0,\\
p^3+p^2&\text{if }\rho\in p^2\Z_q^*,\\
0&\text{if }\rho\in p\Z_q^*,\\
p^3+p^2&\text{if }\rho\in\Z_q^*.
\end{cases}\]
Define $D_{q,\rho}$ as the number of pairs $(u_1,u_2)\in\Z_q^2\times\Z_q^2$ such that $\|u_1-u_2\|\equiv\rho$. Then, we have
\[D_{q,\rho}=\begin{cases}
p^8&\text{if }\rho\equiv0,\\
p^9+p^8&\text{if }\rho\in p^2\Z_q^*,\\
0&\text{if }\rho\in p\Z_q^*,\\
p^9+p^8&\text{if }\rho\in\Z_q^*.
\end{cases}\]
\end{lemma}
After proving these lemmas, we faced a major problem that is much more complicated when considering $\Z_q^2$ rather than $\F_p^2$. To be more specific, in $\F_p^2$, Hanson, Lund, and Roche-Newton proved the following lemma:
\begin{lemma}[Lemma 7, \cite{Hanson2016distinct}]
Given any point pairs $x=(x_1,x_2)\in\mathbb{F}_p^2\times\mathbb{F}_p^2$ and $y=(y_1,y_2)\in\mathbb{F}_p^2\times\mathbb{F}_p^2$ with $\|x_1-x_2\|=\|y_1-y_2\|$, define $\mathcal{N}(x,y)$ to be the number of pairs $(\mathcal{R}_1,\mathcal{R}_2)$ of reflections with $y_i=\mathcal{R}_2\circ\mathcal{R}_1(x_i)$ for $i=1,2$. Then, 
\[\mathcal{N}(x,y)=\begin{cases}p-1&\text{if }x_1-x_2\neq y_1-y_2\text{ and }p\equiv1\mod 4,\\p+1&\text{if }x_1-x_2\neq y_1-y_2\text{ and }p\equiv3\mod 4,\\p&\text{if }x_1-x_2=y_1-y_2\text{ and }\|x_1-y_1\|\neq0,\\0&\text{if }x_1-x_2=y_1-y_2\text{ and }\|x_1-y_1\|=0.\end{cases}\]
\end{lemma}
However, in $\Z_q^2$, we have no criterion for the $\mathcal{N}(x,y)$ value. This problem also explicitly shows how the structure of $\Z_q^2$ differs from the structure of $\F_p^2$. Fortunately, with the help of a computer program, we directly calculated the distribution of $\mathcal{N}(x,y)$ for small $p$ and made the conjecture as follows.
\begin{conjecture}\label{key_conj}
Given any point pairs $x=(x_1,x_2)\in\Z_q^2\times\Z_q^2$ and $y=(y_1,y_2)\in\Z_q^2\times\Z_q^2$, define $\mathcal{N}(x,y)$ to be the number of pairs $(\mathcal{R}_1,\mathcal{R}_2)$ of reflections with $y_i=\mathcal{R}_2\circ\mathcal{R}_1(x_i)$ for $i=1,2$ and let
\[\mathcal{A}_x(n):=\left|\{y\mid\mathcal{N}(x,y)=n\}\right|.\]
Then, we have
\[\begin{array}{lll}
\mathcal{A}_x(p^3-3p^2)=p^9-p^8,&\mathcal{A}_x(p^3-p^2)=p^8,&\mathcal{A}_x(p^3)=p^8-2p^7+p^6,\\
\mathcal{A}_x(p^4-p^3)=p^6-p^5,&\mathcal{A}_x(p^4)=p^5-2p^4+p^3,&\mathcal{A}_x(p^5-p^4)=p^3-p^2,\\
\mathcal{A}_x(p^5)=p^2-2p+1,&\mathcal{A}_x(p^6-p^5)=1,&\mathcal{A}_x(n)=0
\end{array}\]
for all $n$ not mentioned above.
\end{conjecture}
Assume that the above conjecture is true. Now, we introduce two well-known theorems in graph theory and linear algebra that help us to finish the proof of our main result.
\begin{theorem}[Expander Mixing Lemma, \cite{Hanson2016distinct}]\label{EML}
Let $G=(V,E)$ be a $\delta$-regular graph with $|V|=n$, and let $A$ be the adjacency matrix for $G$. Suppose that the absolute values of all but the largest eigenvalue of $A$ are bounded by $\lambda$. Suppose that $f,g\in L^2(V)$. Then, we have
\[|\langle f,Ag\rangle-\delta n\mathbb{E}(f)\mathbb{E}(g)|\leq\lambda\|f\|\|g\|.\]
Especially, for any $S,T\subseteq V$,
\[\left|E(S,T)-\frac{\delta|S||T|}{n}\right|\leq\lambda\sqrt{|S||T|}\]
where $E(S,T)$ is the number of edges between $S$ and $T$.
\end{theorem}
\begin{theorem}[Gershgorin Circle Theorem, \cite{Brualdi1994regions}]\label{GCT}
Let $A=[A_{ij}]$ be an $n\times n$ matrix, and let $r_i=\sum_{j=1}^n|a_{ij}|$ be the sum of the absolute values of the $i$-th row of $A$. Then, for any eigenvalue $\lambda$ of $A$, there is a $1\leq i\leq n$ such that $|\lambda-a_{ii}|\leq r_i$.
\end{theorem}
For simplicity, we declare some notations used in the later context:
\begin{definition}[$\widetilde{P},Q',Q_d',\Pi_d'$]
Given a point set $P\subseteq\Z_q^2$ and $d\in\Z_q^*$, define
\begin{equation*}
\begin{split}
\widetilde{P}&:=\{(x,y)\in P^2\mid x-y=(a_1,a_2)\text{ with }a_1,a_2\in\Z_q^*\},\\
Q'=Q'(P)&:=\{(x,y,z,w)\in\widetilde{P}^2\mid B(x,z)=B(y,w),B(x,z)\text{ is non-isotropic}\},\\
Q_d'=Q_d'(P)&:=\{(x,y,z,w)\in Q'\mid\|x-y\|=\|z-w\|=d\},\\
\Pi_d'=\Pi_d'(P)&:=\{(x,y)\in\widetilde{P}\mid\|x-y\|=d\}.
\end{split}
\end{equation*}
\end{definition}
With this notation and the assumption of Conjecture \ref{key_conj}, we can prove the following proposition, which is the most important part of the proof of the main theorem. A similar result for the $\F_p^2$ case can be found in Proposition 12 in \cite{Hanson2016distinct}.
\begin{proposition}\label{proposition4.2.14}
If Conjecture \ref{key_conj} is correct, for any $d\in\Z_q^*$, we have
\[|Q_d'|\ll\frac{{|\Pi_d'|}^2}{p^3}+p^5|\Pi_d'|.\]
\end{proposition}
\begin{proof}
Let $G=(V,E)$ be a graph with 
\[V:=\{x=(x_1,x_2)\in\Z_q^2\times\Z_q^2\mid\|x_1-x_2\|=d\}\Rightarrow|V|=p^9+p^8\]
and 
\begin{equation*}
\begin{split}
E:=\{\{x,y\}\mid&\hspace{0.5em}x=(x_1,x_2),y=(y_1,y_2)\\
&\hspace{1em}\text{ with }B(x_1,y_1)=B(x_2,y_2)\text{ is non-isotropic}\}.
\end{split}
\end{equation*}
For a vertex $x\in V$, define $\Gamma(x)$ to be the neighborhood of $x$ and notice that for any $x$, $|\Gamma(x)|$ is equal to the number of non-isotropic lines. That is, $G$ is a $(p^6-p^5)$-regular graph. Additionally, notice that
\[|\Gamma(x)\cap\Gamma(y)|=|\{z\in\Gamma(x)\mid y\in\Gamma(z)\}|=\mathcal{N}(x,y).\] 
Let $A$ be the adjacency matrix of $G$. Then, the $(x,y)$-th entry of $A^2$
\[{(A^2)}_{xy}=|\Gamma(x)\cap\Gamma(y)|=\mathcal{N}(x,y).\]
By Conjecture \ref{key_conj}, $A^2$ is a regular matrix. That is, for every row of $A^2$, the sum of the absolute value of the entries in that row is the same. Let $J$ be the all-1s matrix, $I$ be the identity matrix, and $E$ be the error matrix with 
\[A^2=(p^3-3p^2)J+(p^6-p^5-p^3+3p^2)I+E.\]
For the $x$-th row of $E$, the absolute row sum is 
\begin{equation*}
\begin{split}
\sum_{y\in V}|E_{xy}|=&\left|{(A^2)}_{xy}-(p^3-3p^2)-(p^6-p^5-p^3+3p^2)\delta_{x,y}\right|\\
=&2p^2\times p^8+3p^2\times(p^8-2p^7+p^6)\\
&+(p^4-2p^3+3p^2)\times(p^6-p^5)\\
&+(p^4-p^3+3p^2)\times(p^5-2p^4+p^3)\\
&+(p^5-p^4-p^3+3p^2)\times(p^3-p^2)\\
&+(p^5-p^3+3p^2)\times(p^2-2p+1)\\
&+(p^3-3p^2)\times(2p^7-2p^6+2p^4-2p^3+2p-2)\\
=&\Theta(p^{10}).
\end{split}
\end{equation*}
Additionally, every diagonal entry of $E$ vanishes, so for any eigenvalue $\lambda_E$ of $E$, we have 
\begin{equation}\label{eq32}
|\lambda_E|\ll p^{10}
\end{equation}
by Theorem \ref{GCT}.
\par 
Since $A$ is a regular matrix, the all-1s vector is an eigenvector of $A$. Suppose that $v$ is an eigenvector of $A$, which is orthogonal to the all-1s vector and $Av=\lambda v$. Then, 
\[Ev=(A^2-(p^3-3p^2)J-(p^6-p^5-p^3+3p^2)I)v=(\lambda^2-p^6+p^5+p^3-3p^2)v.\]
In other words, $v$ is an eigenvector of $E$ with eigenvalue $\lambda^2-p^6+p^5+p^3-3p^2$. By Equation \ref{eq32}, we know that 
\[|\lambda^2-p^6+p^5+p^3-3p^2|\ll p^{10}\Rightarrow\lambda\ll p^5.\]
\par
Finally, we can use Theorem \ref{EML}. Taking $S=T=\Pi_d'$, we have
\[E(\Pi_d',\Pi_d')\ll\frac{(p^6-p^5){|\Pi_d'|}^2}{p^9+p^8}+p^5|\Pi_d'|\ll\frac{{|\Pi_d'|}^2}{p^3}+p^5|\Pi_d'|.\]
In the end, by definition, we have $|Q_d'|=E(\Pi_d',\Pi_d')$, so we finish the proof.
\end{proof}
Next, the following lemma also demonstrates another difference between $\Z_q^2$ and $\mathbb{F}_p^2$. After the proof of Lemma \ref{lemma4.2.15}, we will discuss the difference in the remark in detail. 
\begin{lemma}\label{lemma4.2.15}
For any point set $\mathcal{P}\subseteq\Z_q^2$, we have
\[\sum_{d\in\Z_q^*}{|\Pi_d'(\mathcal{P})|}^2\ll\frac{{|\mathcal{P}|}^4}{q}+q^3{|\mathcal{P}|}^2.\]
\end{lemma}
\begin{proof}
For a $d\in\Z_q^*$, consider a graph $G_d=(V_d,E_d)$ with 
\[V_d=\Z_q^2\times\Z_q^2\text{ and }E_d=\{\{(x_1,x_2),(y_1,y_2)\}\subset V\mid\|x_1-y_1\|=\|x_2-y_2\|=d\}.\]
Notice that by definition, we have ${|\Pi_d'|}^2=E(\mathcal{P}\times\mathcal{P},\mathcal{P}\times\mathcal{P})$.
Now, we want to use Theorem \ref{EML} to obtain an upper bound of $E(\mathcal{P}\times\mathcal{P},\mathcal{P}\times\mathcal{P})$. First, note that $|V|=q^4$ and $G_d$ is a ${|C_d(0)|}^2={(p^3+p^2)}^2$-regular graph by Lemma \ref{lemma4.2.8}. Let $A$ be the adjacency matrix of $G$. By Theorem \ref{GCT}, every eigenvalue of $A$ is bounded by 
\[\sum_{y\in A}|A_{xy}|={(p^3+p^2)}^2.\]
Combining them together, we obtain that
\begin{equation*}
\begin{split}
&{|\Pi_d'(\mathcal{P})|}^2\ll\frac{{(p^3+p^2)}^2{|\mathcal{P}|}^4}{q^4}+{(p^3+p^2)}^2{|\mathcal{P}|}^2\\
&\Rightarrow\sum_{d\in\Z_q^*}{|\Pi_d'(\mathcal{P})|}^2\ll\frac{{|\mathcal{P}|}^4}{q}+q^3{|\mathcal{P}|}^2.
\end{split}
\end{equation*}
This completes our proof.
\end{proof}
\begin{remark}
It is easy to see that in the above proof, the bound of the second-largest eigenvalue of $A$ is quite bad. ${(p^3+p^2)}^2$ is its largest eigenvalue, which is also the trivial one. However, due to the structure of $G_d$, it is difficult to apply the skill we used in Proposition \ref{proposition4.2.14}. That is, decomposing $A^2$ into the linear combination of $J, I$, and an error matrix $E$ (as Hanson, Lund, and Roche-Newton did in \cite{Hanson2016distinct}) does not help us. For a reason, in short, both the numbers of entries in a given row of $A^2$ equal to 0 and 2, respectively, are about $q^2$. Since they have almost the same amount, we cannot set a proper coefficient for $J$ such that the row sum of the error matrix is properly bounded. Therefore, the technique improves nothing more than coefficients.
\par
As a result, I am looking for another skill to obtain a statement like
\[\sum_{d\in\Z_q^*}{|\Pi_d'(\mathcal{P})|}^2\ll\frac{{|\mathcal{P}|}^4}{q}+q^{3-\epsilon}{|\mathcal{P}|}^2,\forall\,\mathcal{P}\subseteq\Z_q^*\]
for some constant $\epsilon>0$. If such a statement and Conjecture \ref{key_conj} hold, we could get a theorem in the same form as Theorem \ref{theorem4.2.1}:
\begin{quote}
For any $\mathcal{P}\subseteq\Z_q^2$ with $|\mathcal{P}|\gg q^\alpha$, we have $B(\mathcal{P})\gg q^2$ for some non-trivial constant $\alpha$. i.e. $\alpha<2$.
\end{quote}
\end{remark}
Now, with the help of Lemma \ref{lemma4.2.15}, we can prove the following lemma. This lemma is the final step before proving our main goal in this section, Theorem \ref{main}.
\begin{lemma}\label{lemma4.2.16}
If Conjecture \ref{key_conj} is correct, then for any point set $\mathcal{P}\subset\Z_q^2$,
\[|Q'(\mathcal{P})|\ll\frac{{|\mathcal{P}|}^4}{q^2}+q^2{|\mathcal{P}|}^2.\]
\end{lemma}
\begin{proof}
By definition of $Q'$, we know that
\[|Q'|=\sum_{d\in\Z_q^*}|Q_d'|.\]
Together with Proposition \ref{proposition4.2.14} and Lemma \ref{lemma4.2.15}, we get
\begin{equation*}
\begin{split}
|Q'|\ll\sum_{d\in\Z_q^*}\left(\frac{{|\Pi_d'|}^2}{q}+q^{\frac{5}{3}}|\Pi_d'|\right)\ll\frac{{|\mathcal{P}|}^4}{q^2}+q^2{|\mathcal{P}|}^2+q^{\frac{5}{3}}{|\mathcal{P}|}^2\ll\frac{{|\mathcal{P}|}^4}{q^2}+q^2{|\mathcal{P}|}^2.
\end{split}
\end{equation*}
\end{proof}
Finally, everything is ready. Let us demonstrate the proof of our main goal.
\begin{proof}[Proof of Theorem \ref{main}]
For a non-isotropic line $l\in B(\mathcal{P})$, define its multiplicity $w(l)$ to be the number of pairs $(x,y)\in\mathcal{P}^2$ with $B(x,y)=l$. Then we know that 
\[\sum_{l\in B(P)}w(l)=\left|\widetilde{\mathcal{P}}\right|.\]
Since $|\mathcal{P}|\gg q^{\frac{5}{3}}$, $\left|\widetilde{\mathcal{P}}\right|=\Theta\left({|\mathcal{P}|}^2\right)$. Therefore, by Cauchy-Schwarz inequality, we have
\[{|\mathcal{P}|}^4\ll {\left(\sum_{l\in B(\mathcal{P})}w(l)\right)}^2\leq |B(\mathcal{P})|\left(\sum_{l\in B(\mathcal{P})}{w(l)}^2\right)=|B(\mathcal{P})||Q'(\mathcal{P})|.\]
Together with Lemma \ref{lemma4.2.16}, we have
\[|B(\mathcal{P})|\gg\frac{{|\mathcal{P}|}^4}{\frac{{|\mathcal{P}|}^4}{q^2}+q^2{|\mathcal{P}|}^2}\Rightarrow|B(\mathcal{P})|\gg q^{-2}{|\mathcal{P}|}^2\gg q^{2(m-1)}.\]
\end{proof}

\section*{Appendix A. Proof of Szemer\'{e}di-Trotter Theorem}
\begin{recall*}[Szemer\'{e}di-Trotter Theorem]
On Euclidean plane $\mathbb{R}^2$, given a point set $\mathcal{P}$ and a line set $\mathcal{L}$, then the number of their incidences
\[\mathcal{I}(\mathcal{P},\mathcal{L})=O\left({|\mathcal{P}|}^{\frac{2}{3}}{|\mathcal{L}|}^{\frac{2}{3}}+|\mathcal{P}|+|\mathcal{L}|\right).\]
\end{recall*}
Theorem \ref{Szemeredi-Trotter} is first proved by Szemer\'{e}di and Trotter in 1983. But the technique is so complicated, so we decide to use another way to prove it. The following proof is provided by Sz\'{e}kely (1997, \cite{Szekely1997crossing}).
\begin{proof}[Proof of Theorem \ref{Szemeredi-Trotter}]
First, we may discard lines that contain two or fewer points since they can provide at most $2|\mathcal{L}|$ incidences in total. For any line $l\in\mathcal{L}$ containing $k_l$ points in $\mathcal{P}$, these $k_l$ points will cut the line into $k_l-1$ segments. 
\par
Consider a graph $G=(V, E)$ with $V=\mathcal{P}$ and for any $x,y\in V$, $\{x,y\}\in E$ if and only if $x,y$ are two endpoints of a segment mentioned above. Notice that any two lines intersect in at most one point, so the crossing number of $G$ is at most $\frac{|\mathcal{L}|(|\mathcal{L}|-1)}{2}$. By the crossing number inequality, we know that 
\begin{equation}\label{eq33}
\frac{|\mathcal{L}|(|\mathcal{L}|-1)}{2}\geq\frac{{|E|}^3}{64{|V|}^2}\text{ or }|E|\leq4|\mathcal{P}|.
\end{equation}
By construction, we know that 
\[|E|=\sum_{l\in\mathcal{L}}(k_l-1)\geq\frac{1}{2}\sum_{l\in\mathcal{L}}k_l=\frac{1}{2}\mathcal{I}(\mathcal{P},\mathcal{L}).\]
Together with Equation \eqref{eq33}, we have 
\[\mathcal{I}(\mathcal{P},\mathcal{L})\leq2|E|\ll|\mathcal{P}|+{|\mathcal{L}|}^{\frac{2}{3}}{|\mathcal{P}|}^{\frac{2}{3}}.\]
This completes our proof.
\end{proof}
\section*{Appendix B. Proof of Szemer\'{e}di-Trotter Type Theorem in Finite Field}
\begin{recall*}[Szemer\'{e}di-Trotter type Theorem in $\mathbb{F}_p$, \cite{Stevens2017improved}]
On $\mathbb{F}_p^2$ where $p$ is a prime, given a point set $\mathcal{P}=A\times B$ and a line set $\mathcal{L}$ with 
\[|A|\leq|B|,|A|{|B|}^2\leq {|\mathcal{L}|}^3\text{, and }|A||\mathcal{L}|\ll p^2\]
then the number of their incidences
\[\mathcal{I}(A\times B,\mathcal{L})=O\left({|A|}^{\frac{3}{4}}{|B|}^{\frac{1}{2}}{|\mathcal{L}|}^{\frac{3}{4}}+|\mathcal{L}|\right).\]
\end{recall*}
\begin{proof}[Proof of Theorem \ref{Szemeredi-Trotter_in_Fp}]
First, because of the order of the upper bound, we can modify $\mathcal{L}$ as follows: (In other words, we can remove or add these lines without affecting the correctness of the statement.)
\begin{enumerate}
\item Remove all vertical lines in $\mathcal{L}$:\\
For those vertical lines in $\mathcal{L}$, they at most contribute $|A||B|$ incidences. Since $|A|{|B|}^2\leq {|\mathcal{L}|}^3$, $|A||B|\leq {|A|}^{\frac{3}{4}}{|B|}^{\frac{1}{2}}{|\mathcal{L}|}^{\frac{3}{4}}$. Thus, the incidence made by vertical lines will not play a crucial role.
\item We may assume ${|B|}^2\leq|A||\mathcal{L}|$:\\
Since there is no vertical line, we have $\mathcal{I}(A\times B,\mathcal{L})\leq|A||\mathcal{L}|$. Assume ${|B|}^2>|A||\mathcal{L}|$. Then it is clear that
\[\mathcal{I}(A\times B,\mathcal{L})\leq|A||\mathcal{L}|\leq {|A|}^{\frac{3}{4}}{|B|}^{\frac{1}{2}}{|\mathcal{L}|}^{\frac{3}{4}}\]
and thus the Theorem \ref{Szemeredi-Trotter_in_Fp} holds.
\item At most ${|A|}^{\frac{1}{2}}{|\mathcal{L}|}^{\frac{1}{2}}$ lines are parallel or concurrent (i.e. passing the same point):\\
We iteratively remove a set of parallel or concurrent lines with a size greater than ${|A|}^{\frac{1}{2}}{|\mathcal{L}|}^{\frac{1}{2}}$. Suppose that in $i$-th step, we remove $n_i$ lines. Then, these $n_i$ lines contribute at most $|A||B|+n_i$ incidences and we need at most $\frac{|\mathcal{L}|}{{|A|}^{\frac{1}{2}}{|\mathcal{L}|}^{\frac{1}{2}}}=\frac{{|\mathcal{L}|}^{\frac{1}{2}}}{{|A|}^{\frac{1}{2}}}$ steps. After these steps, we remove at most
\[\frac{{|\mathcal{L}|}^{\frac{1}{2}}}{{|A|}^{\frac{1}{2}}}\times|A||B|+\sum n_i\ll{|A|}^{\frac{3}{4}}{|B|}^{\frac{1}{2}}{|\mathcal{L}|}^{\frac{3}{4}}+|\mathcal{L}|\]
incidences since ${|B|}^2\leq|A||\mathcal{L}|$.
\end{enumerate}
Since $\mathcal{L}$ has no vertical lines, the affine dual 
\[\mathcal{L}^*:=\{(c,d)\in\mathbb{F}_p^2\mid l:y=cx+d\in\mathcal{L}\}\]
of $\mathcal{L}$ is well-defined. Then, we have the relation
\[\mathcal{I}(\mathcal{P},\mathcal{L})=\left|\{(a,b,c,d)\in A\times B\times\mathcal{L}^*\mid b=ca+d\}\right|.\]
If we set 
\[E:=\{(a,c,d,a',c',d')\in {(A\times\mathcal{L}^*)}^2\mid ca+d=c'a'+d'\},\]
then by Cauchy-Schwarz inequality and the above relation, we have
\[\mathcal{I}(\mathcal{P},\mathcal{L})\leq {|B|}^{\frac{1}{2}}{|E|}^{\frac{1}{2}}.\]
Next, we will bound $|E|$ with Theorem \ref{theorem3.3.1}. Define a point set and a plane set by 
\[\mathcal{P}'=\{(a,c,d)\in A\times\mathcal{L}^*\},\Pi=\{\pi:cx+d=ay+z\mid a\in A,(c,d)\in\mathcal{L}^*\}.\]
Then, we have $|\mathcal{P}'|=\Pi=|A||\mathcal{L}|$ and $|E|=\mathcal{I}(\mathcal{P}',\Pi)$. Now, via our modification above, $\mathcal{P}'$ and $\mathcal{L}$ satisfy the condition mentioned in Theorem \ref{theorem3.3.1} with 
\[k\leq {|A|}^{\frac{1}{2}}{|\mathcal{L}|}^{\frac{1}{2}}.\]
Hence, by Theorem \ref{theorem3.3.1}, 
\[\mathcal{I}(\mathcal{P}',\Pi)\ll {|A|}^{\frac{3}{2}}{|\mathcal{L}|}^{\frac{3}{2}}.\]
To sum up, we know that
\[\mathcal{I}(\mathcal{P},\mathcal{L})\leq {|B|}^{\frac{1}{2}}{|E|}^{\frac{1}{2}}\ll {|A|}^{\frac{3}{4}}{|B|}^{\frac{1}{2}}{|\mathcal{L}|}^{\frac{3}{4}},\]
which completes the proof.
\end{proof}
\section*{Appendix C. Proofs of Lemmas Describing Properties about \texorpdfstring{$\Z_q$}{}}
In this Appendix, we will introduce Lemma C.1 to Lemma C.3 and use them to prove Lemma \ref{lemma4.2.5} to Lemma \ref{lemma4.2.8} in Section 4.2.
\begin{aplemma}[C.1]\label{lemmaC.1}
Suppose that $x\in\Z_q^2$ and $\mathcal{S}$ is a reflection that does not fix $x$. Then the fixed line of $\mathcal{S}$ is $B(x,\mathcal{S}(x))$. Moreover, a good line $l$ is the fixed line of a unique reflection if and only if it is non-isotropic. If $y\in\Z_q^2$ is any point such that $\|x-y\|\not\equiv0$ and the line passing $x$ and $y$ is good, then $B(x,y)$ is non-isotropic, and there is a unique reflection $\mathcal{S}$ such that $\mathcal{S}(x)=y$ which fixes $B(x,y)$. 
\end{aplemma}
\begin{proof}
Observe that if $u$ is fixed by $\mathcal{S}$ then 
\[\|x-u\|=\|\mathcal{S}(x)-\mathcal{S}(u)\|=\|\mathcal{S}(x)-u\|\]
so that $u\in B\big(x,\mathcal{S}(x)\big)$. Additionally, the fixed part of $\mathcal{S}$ formed a line, so the fixed part of $\mathcal{S}$ is $B\big(x,\mathcal{S}(x)\big)$. 
\par
Let $u_1$ and $u_2$ be any distinct points on the line $l$, which is assumed to be non-isotropic. Set $d=(d_1,d_2):=u_1-u_2$. The reflection $\mathcal{S}$ by 
\[\frac{1}{d_1^2+d_2^2}\begin{pmatrix}d_1^2-d_2^2&2d_1d_2\\2d_1d_2&d_2^2-d_1^2\end{pmatrix}\]
about $u_1$ fixes $l$. If there is another reflection $\mathcal{S}'$ fixing $l$, then $\mathcal{S}\circ\mathcal{S}'$ would be either a rotation or a translation and $\mathcal{S}\circ\mathcal{S}'$ fixes a line, so $\mathcal{S}=\mathcal{S}'$. In other words, such reflection is unique. 
\par
Finally, suppose that $y\in\Z_q^2$ is distinct from $x$ with the line $l$ passing $x$ and $y$ being good. Since $p\equiv3\mod4$, $B(x,y)$ is always non-isotropic.
\end{proof}
\begin{aplemma}[C.2]\label{lemmaC.2}
Let $x,y\in C_\rho(u)$ for elements $x,y,u\in\Z_q^2$ and $\rho\in\Z_q^*$. There is a unique rotation $\mathcal{R}$ fixing $u$ and sending $x$ to $y$.
\end{aplemma}
\begin{proof}
After applying a translation, we may assume $u=0$. Then, we have $\|x\|=\|y\|\in\Z_q^*$. Let $x=(x_1,x_2)$ and $y=(y_1,y_2)$. Solve a rotation matrix $R$ with
\[Rx=y\Rightarrow\begin{pmatrix}a&-b\\b&a\end{pmatrix}\begin{pmatrix}x_1\\x_2\end{pmatrix}=\begin{pmatrix}y_1\\y_2\end{pmatrix}\Rightarrow\begin{pmatrix}x_1&-x_2\\x_2&x_1\end{pmatrix}\begin{pmatrix}a\\b\end{pmatrix}=\begin{pmatrix}y_1\\y_2\end{pmatrix}.\]
Since $\|x\|\in\Z_q^*$, there is a unique solution $(a,b)$. Additionally, 
\[a^2+b^2=\frac{1}{\rho}\left({(ax_1-bx_2)}^2+{(bx_1+ax_2)}^2\right)=\frac{1}{\rho}(y_1^2+y_2^2)=1.\]
Thus, such a rotation exists. If there is another rotation $R'$ with the same property, then the rotation $R^{-1}\circ R'$ fixes $x$ and $u$, so $R=R'$. That is, such rotation is unique.
\end{proof}
\begin{aplemma}[C.3]\label{lemmaC.3}
Suppose that $x,y,z,w\in\Z_q^2$ such that $(x,y)\neq(z,w)$ and 
\[\|x-y\|=\|z-w\|\in\Z_q^*.\]
If $x-y\neq z-w$, then there is a unique rotation $\mathcal{R}$ with $\mathcal{R}(x)=z$ and $\mathcal{R}(y)=w$. If $x-y=z-w$, then there is no rotation $\mathcal{R}$ with $\mathcal{R}(x)=z$ and $\mathcal{R}(y)=w$.
\end{aplemma}
\begin{proof}
Divide it into two cases:
\begin{enumerate}
\item Assume $x-y\neq z-w$. Let $\mathcal{T}$ be the translation by $z-x$. i.e. $\mathcal{T}(x)=z$. Note that 
\[\|z-w\|=\|x-y\|=\|\mathcal{T}(x)-\mathcal{T}(y)\|=\|z-\mathcal{T}(y)\|.\]
Thus, $\mathcal{T}(y)$ and $w$ are on a common circle with an invertible radius and are centered at $z$. By Lemma \ref{lemmaC.2}, there is a non-trivial rotation $\mathcal{R}$ with $\mathcal{R}\circ\mathcal{T}(y)=w$. Then $\mathcal{R}'=\mathcal{R}\circ\mathcal{T}$ is the desired rotation. Additionally, if there is another non-trivial rotation $\mathcal{R}''$ with $\mathcal{R}''(x)=z$ and $\mathcal{R}''(y)=w$.
Then, ${\mathcal{R}'}^{-1}\circ\mathcal{R}''$ is a rotation fix $x,y$. In other words, the rotation is unique.
\item Assume $x-y=z-w$. Similarly, let $\mathcal{T}$ be the translation by $z-x$. i.e. $\mathcal{T}(x)=z$ and $\mathcal{T}(y)=w$. Assume that $\mathcal{R}$ is a rotation with $\mathcal{R}(x)=z$ and $\mathcal{R}(y)=w$. Then, $\mathcal{R}^{-1}\circ\mathcal{T}$ is a non-trivial rotation fixing $x,y$, which leads to a contradiction.
\end{enumerate}
\end{proof}
\begin{proof}[Proof of Lemma \ref{lemma4.2.5}]
Let $x=x_1p^2+x_2p+x_3$ with $x_i\in\F_p$. Notice that 
\[x^2\equiv(2x_1x_3+{x_2}^2)p^2+2x_2x_3p+{x_3}^2\mod q.\]
Divide it into six cases:
\begin{enumerate}
\item If $\legendre{d}{p}=1$, then there are two nonzero solutions to ${x_3}^2\equiv d\mod p$. Next, fixed $x_3$, there is a unique solution for $x_2$ to 
\[2x_2x_3+\left\lfloor\frac{{x_3}^2}{p}\right\rfloor\equiv\left\lfloor\frac{d}{p}\right\rfloor\mod p.\]
In the end, fixed $x_2,x_3$, there is also a unique solution for $x_1$ to 
\[2x_1x_3+{x_2}^2+\left\lfloor\frac{2x_2x_3p+{x_3}^2}{p^2}\right\rfloor\equiv\left\lfloor\frac{d}{p^2}\right\rfloor\mod p.\]
To sum up, there are two solutions to $x^2\equiv d\mod q$.
\item If $\legendre{d}{p}=-1$, then there is no solution to ${x_3}^2\equiv d\mod p$ and thus $Q(d)=0$.
\item If $p\mid d$ and $p^2\nmid d$, then note that 
\[x^2\equiv2x_2x_3p+{x_3}^2\equiv d\mod p^2.\]
Since $x_3=0$, $d\equiv0\mod p^2$. This results in a contradiction, so $Q(d)=0$.
\item If $p^2\mid d$ and $\legendre{d/p^2}{p}=1$, then solve ${x_3}^2\equiv d\mod p$. We know that $x_3=0$. Also, solve
\begin{equation*}
\begin{split}
&2x_1x_3+{x_2}^2+\left\lfloor\frac{2x_2x_3p+{x_3}^2}{p^2}\right\rfloor\equiv\frac{d}{p^2}\mod p\\
&\Rightarrow {x_2}^2\equiv\frac{d}{p^2}\mod p
\end{split}
\end{equation*}
This has 2 solutions for $x_2$; for both values, $x_1$ can take arbitrary value in $\F_p$, so there are $2p$ solutions in total.
\item If $p^2\mid d$ and $\legendre{d/p^2}{p}=-1$, then consider
\begin{equation*}
\begin{split}
&2x_1x_3+{x_2}^2+\left\lfloor\frac{2x_2x_3p+{x_3}^2}{p^2}\right\rfloor\equiv\frac{d}{p^2}\mod p\\
&\Rightarrow {x_2}^2\equiv\frac{d}{p^2}\mod p
\end{split}
\end{equation*}
This has no solution for $x_2$ in $\F_p$. 
\item If $d=0$, then consider
\[(2x_1x_3+{x_2}^2)p^2+2x_2x_3p+{x_3}^2\equiv0\mod q.\]
Then, we have $x_3=0,x_2=0$, and $x_1$ can be arbitrary elements in $\F_p$, so there are $p$ solutions in total.
\end{enumerate}
\end{proof}
\begin{proof}[Proof of Lemma \ref{lemma4.2.6}]
Assume that the bisector $B:ax+by=c$. After rescaling and shifting, we may also assume that $c=0$ and $a^2+b^2=1$. Let $\mathcal{S}$ be the reflection according to $B$. Then, we know that $x=\mathcal{S}(z)$ and $y=\mathcal{B}(w)$, so 
\[\|x-y\|=\|\mathcal{S}(z)-\mathcal{S}(w)\|=\|\mathcal{S}(z-w)\|=\|z-w\|.\]
\end{proof}
\begin{proof}[Proof of Lemma \ref{lemma4.2.7}]
For a good line $l$, its parametric form is $u+t(1,a)$ with $t\in\mathbb{Z}_q$ and $a\in\Z_q$ is a unit. Since there are $p^3-p^2$ units in $\Z_q$, there are $p^3-p^2$ good lines. Furthermore, since $p\equiv3\mod4$, there is no isotropic line.
\end{proof}
\begin{proof}[Proof of Lemma \ref{lemma4.2.8}]
Divide it into three cases:
\begin{enumerate}
\item Suppose that $\rho=0$. Consider the equation
\[x^2+y^2\equiv0\text{ where }x,y\in\Z_q.\]
Let $x=x_1p^2+x_2p+x_3$ and $y=y_1p^2+y_2p+y_3$ with $x_i,y_i\in\F_p$. Then, we get $x_3\equiv y_3\equiv0$, $x_2\equiv y_2\equiv0$, and $x_1,y_1$ can be arbitrary elements in $\F_p$. Thus,
\[|C_0(u)|=p^2\]
\item Suppose that $\rho=n_1p^2+n_2p$ with $n_i\in\F_p$ and at least one of them are non-zero. Consider the equation
\[x^2+y^2\equiv\rho\text{ where }x,y\in\Z_q.\]
Let $x=x_1p^2+x_2p+x_3$ and $y=y_1p^2+y_2p+y_3$ with $x_i,y_i\in\F_p$. Then, we get $x_3\equiv y_3\equiv0$, and thus $n_2$ must be 0. To sum up, 
\[|C_\rho(u)|=\begin{cases}p^3+p^2&\text{if }n_2\equiv0\\0&\text{if }n_2\not\equiv0.\end{cases}\]
\item Suppose that $\rho=n_1p^2+n_2p+n_3$ with $n_3\not\equiv0$. Consider the equation
\[x^2+y^2\equiv\rho\text{ where }x,y\in\Z_q.\]
Let $x=x_1p^2+x_2p+x_3$ and $y=y_1p^2+y_2p+y_3$ with $x_i,y_i\in\F_p$. Then, there are $p+1$ pairs $(x_3,y_3)$ satisfying 
\[{x_3}^2+{y_3}^2\equiv n_3\mod p;\]
there are $p$ pairs $(x_2,y_2)$ satisfying 
\[x_2x_3+y_2y_3+\left\lfloor\frac{{x_3}^2+{y_3}^2}{p}\right\rfloor\equiv n_2\mod p\] if the rest coefficients are fixed (and $(x_3,y_3)\not\equiv(0,0)$); there are $p$ pairs $(x_1,y_1)$ satisfying 
\[2x_1x_3+2y_1y_3+{x_2}^2+{y_2}^2+\left\lfloor\frac{2x_2x_3p+2y_2y_3p+{x_3}^2+{y_3}^2}{p^2}\right\rfloor\equiv n_1\mod p\]
if the rest coefficients are fixed (and $(x_3,y_3)\not\equiv(0,0)$). To sum up, 
\[|C_\rho(u)|=p^3+p^2.\]
\end{enumerate}
\end{proof}

\end{document}